\pdfoutput=1
\documentclass[10pt,conference]{IEEEtran}
\IEEEoverridecommandlockouts

\usepackage{amsmath,amssymb,amsfonts,amsthm}
\usepackage{algorithmic}
\usepackage{graphicx}
\usepackage{textcomp}
\usepackage{xcolor}
\usepackage{macros}
\def\BibTeX{{\rm B\kern-.05em{\sc i\kern-.025em b}\kern-.08em
    T\kern-.1667em\lower.7ex\hbox{E}\kern-.125emX}}
\begin{document}

\title{A computational method for left-adjointness%
  \thanks{This study was supported by the French project EMERGENCE ReaLiSe and
    the French ANR project EPIQ (ANR-22-PETQ-0007).}%
}

\author{\IEEEauthorblockN{Simon Forest}
\IEEEauthorblockA{Aix--Marseille Univ, CNRS, LIS, Marseille, France\\
\texttt{simon.forest@normalesup.org}}
}


\maketitle

\begin{abstract}
  In this work, we investigate an effective method for showing that functors
between categories are left adjoints. The method applies to a large class of
categories, namely locally finitely presentable categories, which are ubiquitous
in practice and include standard examples like $\Set$, $\Grp$, \etcend Our
method relies on a known description of these categories as orthogonal
sub-classes of presheaf categories. The functors on which our method applies are
the ones that can be presented as particular profunctors, called \emph{Kan
models} in this context. The method for left-adjointness then relies on
computing that a particular criterion is satisfied. From this method, we also
derive another method for showing that a category is cartesian closed. As proofs
of concept and effectivity, we give a concrete implementation of the structures
and of the left-adjointness criterion in \textsc{OCaml} and apply it on several
examples.

\end{abstract}


\section*{Introduction}

Over the years, category theory has proven to be an efficient tool to formalise
arguments and constructions in a mathematical setting. First developed in the
context of algebra and topology by Eilenberg and MacLane in the 1940s, it
quickly expanded to other fields than mathematics. In computer science notably,
category theory found several useful applications in programming languages,
semantics, logic, and other topics. For example: the notion of \emph{monad}
allows modeling the effects of programming languages~\cite{MOGGI199155}, and is
at the heart of the \textsc{Haskell} programming language~\cite{haskellws};
cartesian closed categories provide models of simply-typed
lambda-calculus~\cite{lambek1988introduction}; star-autonomous categories
provide models of linear logic~\cite{seely1987linear}; \etc




The approach of category theory can be considered \emph{abstract}, in the sense
that the emphasis is not put on the concrete definitions of the constructions
being considered but by the universal properties they satisfy. For example,
while the cartesian product of two sets $A$ and $B$ is often \emph{concretely}
defined in set theory as \emph{the} set $A \times B \defeq
\set{\set{\set{a},\set{a,b}} \mid a \in A, b \in B}$, in category theory, it is
considered as \emph{a} set $A \times B$, together with functions $l \co A \times
B \to A$ and $r \co A \times B \to B$, satisfying the following universal
property: for every set $X$ and every pair of function $f \co X \to A$ and $g
\co X \to B$, there is a unique function $h \co X \to A \times B$ such that $f =
l \circ h$ and $g = r \circ h$. This categorical definition of the cartesian
product is just an instance of the definition of the more general notion of
\emph{product} in the category $\Set$ of sets and functions, but can be
expressed in any category. Other standard categorical constructions include:
(co)limits, exponentials, left and right adjoints, (co)ends, left and right Kan
extensions, \etcend
Using this abstract approach, one can go a long way using the toolbox of
category theory to make abstract constructions or prove properties using generic
arguments, without actually considering concrete definitions of the objects
being considered.

The abstraction of category theory raises the question of computation: despite
most constructions being abstractly defined, is it possible to compute them or
reason about them computationally? More generally, to which extent category
theory can be made computable?

In order to computationally manipulate or reason about some mathematical
objects, one first need to be able to represent them, or \emph{encode} them, in
a form that a program can understand, typically as instances of datatypes. Once
such an encoding is given, one can start to wonder what can be computed from it.
For example, one can easily encode finite sets and functions between them using
standard computational structures. With such an encoding, one can compute
several categorical constructions on finite sets, like (co)products, pullbacks,
pushouts, and more generally, any finite limit or colimit. But this encoding
only works for the category $\Set$ and its objects that are finite sets, and
does not directly extend either for infinite sets or other categories, which can
seem limited.

In order to consider other kind of computations in a categorical setting, an
idea is to change the encoding for one similar to the ones used in the context
of algebraic structures. Indeed, computations on groups, rings, \etc often rely
on \emph{presentations} of such structures, which specify a set of generators
and a set of equations satisfied by these generators. When such presentations
are \emph{finite}, one can encode them and give them as input to a computer.
Despite the finiteness, a lot of structures of interest can be described through
finite presentations, and several computations on these structures can be
carried out through these presentations. In group theory for example, one can
consider the word problem on a group presented by a finite presentation; one can
also use the Todd–Coxeter algorithm to enumerate the cosets of a group described
by a finite presentation~\cite{ToddCoxeter1936}; \etcend

One can similarly consider categories described by finite presentations and
investigate the computations that can be done through this encoding. Several
earlier works have taken this approach with successful results. But the
categories described this way are very specific: they have a finite number of
objects (specified by the finite presentation) and a countable number of
morphisms, and are defined up to isomorphism. Standard examples of categories,
like $\Set$, the category of groups $\Grp$, \etc, on the other hand, have an
infinite number of objects and morphisms, which do not even form a set, and are
morally defined up to equivalence of categories. So that these standard examples
cannot be described with finite presentations, contrary to the situation of
other algebraic structures, where most examples of interest are.

In order to encode these giant structures, one can use a more creative notion of
presentation. In the same way that finitely presentable categories form a
particular class of categories that can be presented by a finite presentation as
discussed above, there is a class of large well-behaved categories, called
\emph{locally finitely presentable (l.f.p.) categories}, that can be fully
described by a smaller data than its whole classes of objects and morphisms.
\Lfp categories include a lot of very common categories, whose object can be
described as models of some sort of equational theories called \emph{essentially
  algebraic theories}~\cite{freyd1972aspects,adamek1994locally}: the categories
of groups, rings, graphs, small categories, \etc are classical examples. This
notion of l.f.p. categories relies on the former notion of finite presentation
that we discussed: it requires that every object of these categories can be
expressed as a particular colimits of finitely presented objects. This way,
every object is characterised by the sets of morphisms from finitely presented
objects to it. By some theoretical developments, it implies that objects of an
l.f.p. category $\cC$ can be modeled as particular presheaves over some category
$C$, satisfying some \emph{orthogonality conditions}. When $C$ is finite, it can
be encoded and, under similar finiteness conditions, the orthogonality
conditions can be encoded as well, so that the category $\cC$ can be encoded, in
a sense that there is an inhabitant of a datatype that describes $\cC$ up to
equivalence of categories. Moreover, we are able to computationally represent
the objects of $\cC$ that are equivalent to finite presheaves.

Now it might appear that this new kind of encoding only allows one to generalise
the computations that were possible on finite sets, like finite limits and
colimits, to more general \lfp categories and their objects that can be
described by finite presheaves. While this is still valuable, it still feels
limited compared to more general constructions or problems of category theory.
Indeed, category theory is often concerned about \emph{global} properties of
categories (does this category have limits and colimits of a particular shape?
Is it cartesian closed? \etc), or of functors between categories (is this
functor preserving limits or colimits of a particular shape? Does it have an
adjoint? \etc). While the above encoding is able to encode classical (large)
examples of categories, it still seems that we are not able to do much more than
computing finite limits and colimits, like for finite sets. Also, we still miss
a way to encode another important notion of category theory, namely the one of
functor between categories.

The goal of this paper is to push the limitations of this encoding, in order to
extend the scope of what can be considered computable in category theory. For
this purpose, based on the discussed encoding of \lfp categories, we show that
functors between such categories can be encoded and that the problem of left
adjointness, that is, whether a functor is a left adjoint or not, can be given a
computational treatment.

\paragraph*{Outline and contributions}

In \Cref{sec:computational-model-categories}, we recall some basic definitions
about \lfp categories and in particular discuss how they can be described as
full sub-categories of orthogonal objects of presheaf categories, through a
notion that we call \emph{presheaf model}. In \Cref{sec:modeling-functors}, we
show how one can use these presheaf models to represent functors between \lfp
categories, using another structure called \emph{Kan model}. In
\Cref{sec:la-through-modeling}, we consider the problem of left adjointness of
functors between two \lfp categories and give a sufficient criterion based on
presheaf and Kan models. Checking this criterion will rely on a construction
called \emph{reflection}, that we will study in detail to investigate its
computability. While reflection is not always fully computable, finite
approximations are often enough to decide whether the criterion for left
adjointness is satisfied. The computations of these approximations is organised
into some sort of game called \emph{game of reflection}, that one needs to play
to show the criterion. The exploration of the moves that can be played to reach
a winning configuration can be delegated to a computer, resulting in a
computational method to show left adjointness.

As a proof of concept, we give a first implementation in \ocaml of the method,
discussed in \Cref{sec:implementation}, and apply it to several examples.

\paragraph*{Related works} This work can be considered to belong to the greater
field of \emph{computational category theory}, of which several earlier works
are relevant to this current work. Notably, the different
works~\cite{carmody1991computing,CARMODY1995459,Walters_1992,bush2003computing}
on the computation of certain left Kan extensions, first motivated as a
generalisation of the Todd--Coxeter procedure coming from computational group
theory. They were applied, together with other computational aspects of category
theory, in the context of databases in different
works~\cite{spivak2015relational,schultz2016algebraic,schultz2017algebraic,meyers2022fast}.
A more recent line of works discussed the computational aspects of category
theory in the context of homological
algebra~\cite{posur2021constructive,Posur2019MethodsOC}. Several tools were also
developed to deal with different aspects of category
theory~\cite{vicary2018globular,heidemann10homotopy,finstermimram17,lynch2024GAT},
notably the \texttt{AlgebraicJulia} framework and its \texttt{GATlab}
library~\cite{lynch2024GAT}, which allows one to consider categories of models
of \emph{Generalised Algebraic Theories}~\cite{cartmell1986generalised}, which
are strongly connected to the description of \lfp categories that we consider in
this work.


\section{L.f.p. categories, expressed with presheaves}
\label{sec:computational-model-categories}

In this section, we recall the notion of \emph{locally finitely presentable}
(abbreviated \emph{l.f.p.}) categories, which defines a large class of
well-behaved categories, suited for the development of methods that are amenable
to computations. These categories are often described as categories of models of
\emph{essentially algebraic theories}, making them ubiquitous in practice. The
methods we will develop will rely on a particular characterisation of l.f.p.
categories as orthogonal sub-classes of presheaf categories that we will recall.
We moreover illustrate this characterisation on several examples. We defer the
computational treatment of this description of categories to
\Cref{sec:implementation}, focusing here on the mathematical modeling in greater
generality.

\subsection{Finite presentability}

The theory of \lfp categories is centered on the notion of finitely
presentability, itself derived from the notion of directed colimit. See
\cite{adamek1994locally} for details.
\begin{definition}
  A partial order $(D,\le)$ is \index{directed}\emph{directed} when $D \neq
  \emptyset$ and for all $x,y \in D$, there exists~$z \in D$ such that~${x \le
    z}$ and~${y \le z}$. A small category~$I$ is called \emph{directed} when it
  is isomorphic to a directed partial order~$(D,\le)$. A \emph{directed colimit}
  in a category $\cC$ is a colimit taken on a diagram $d \co I \to \cC$ where
  $I$ is directed.
\end{definition}

\begin{definition}
  An object $P$ of a category $\cC$ is \emph{finitely presentable} (abbreviated
  \fp) when its hom-functor $\Hom(P,-) \co \cC \to \Set$ commutes with directed
  colimits.
\end{definition}
Concretely, it means that for every directed colimit $(p_i \co d(i) \to L)_{i
  \in \Ob(I)}$ on a diagram $d \co I \to \cC$ and every morphism $f \co P \to
L$, there is $i \in \Ob(I)$ and $u \co P \to d(i)$ such that $f = p_i \circ u$,
and this pair $(i,u)$ is \emph{essentially unique}: given another such pair
$(i',u')$, there is $j \in \Ob(I)$ with $i,i' \le j$ such that $d(i \le j) \circ
u = d(i' \le j) \circ u'$.
\begin{example}
  When $\cC = \Set$, the \fp objects are the finite sets. In the case of $\cC =
  \Grp$, the category of groups, the \fp objects are the groups which admits a
  presentation through a finite number of generators and a finite number of
  equations. More generally, the \fp objects in other categories are often the
  objects which can be presented using a notion of finite presentation.
\end{example}

\subsection{\Lfp categories}

\begin{definition}[{\cite[Definition 1.9]{adamek1994locally}}]
  \label{def:lfp-categories}
  A locally small category $\cC$ is \emph{locally finitely presentable} (or
  \lfp) when
  \begin{enumerate}[(i)]
  \item it has all small colimits;
  \item every object $X \in \cC$ is a directed colimit of finitely presentable objects;
  \item the full subcategory over the finitely presentable objects of $\cC$ is
    \emph{essentially small}, that is, equivalent to a small category.
  \end{enumerate}
\end{definition}

\begin{example}
  We can already observe that $\Set$ is an \lfp category: indeed, every set is a
  directed union (in particular, a directed colimit) of its finite subsets, and
  the full sub-category of finite subsets is essentially small.
\end{example}
\Cref{def:lfp-categories} is a theoretical but rather abstract definition of
\lfp categories, that hides their relevance in practical situations. One usually
prefer to recognise \lfp categories as models of \emph{essentially algebraic
  theories}~\cite{freyd1972aspects,adamek1994locally}. The latter are an
extension of standard algebraic theories for algebraic structures that allows
for operations with restricted domains. See \Cref{sec:eat} for details.
\begin{example}
  The categories of $\Mon$ of monoids, $\Grp$ of groups, $\Grph$ of graphs,
  $\Cat$ of small categories, \etc are all categories of models of essentially
  algebraic theories. For example, $\Mon$ is the category of models of a theory
  with one sort $M$, a constant $e$ and an operation $m \co M\times M \to M$
  satisfying the equations $m(e,x) = x$, $m(x,e) = x$ and $m(m(x,y),z) =
  m(x,m(y,z))$ for any $x,y,z$ of $M$. Similarly for $\Grp$, $\Grph$ and $\Cat$.
  We deduce that they are all \lfp categories by a well-known characterisation
  (see \Cref{thm:lfp-cats-as-eat-models}).
\end{example}

\subsection{Presheaves on a category}

The methods we develop in this article rely on a description of \lfp categories
as particular sub-classes of presheaf categories. Presheaves are particularly
relevant for the computational methods we investigate in this article, since
they are rather easy to encode. We recall here some basic definitions and facts
about presheaves and presheaf categories. Let $C$ be a small category.
\begin{definition}
  A \emph{presheaf} on $C$ is a functor $X \co \catop C \to \Set$. A
  \emph{morphism} between two presheaves $X$ and $Y$ is a natural transformation
  $\alpha \co X \To Y$. We write $\ps C$ for the \emph{presheaf category} of
  presheaves on $C$ and presheaf morphisms.
\end{definition}
A lot of naturally occurring structures can be described as presheaves on an
adequate category:
\begin{example}
  Considering the category $G$ with two objects $N$ and $A$ and two non-trivial
  morphisms $\csts,\cstt \co N \to A$, we can consider the presheaf category
  $\ps G$, which can be equivalently described as the category of directed
  graphs and their morphisms.
\end{example}
\noindent Presheaf categories have a lot of good properties, notably:
\begin{proposition}
  \label{prop:pscat-props}
  Given a small category $C$, $\ps C$ is \lfp, complete, cocomplete and
  cartesian closed. Moreover, the colimits and limits are pointwisely computed
  as in $\Set$: given a diagram of presheaves $X_{(-)} \co I \to \ps C$ and $c
  \in C$, we have $(\colim_{i \in I} X_i)(c) \cong \colim_{i \in I} (X_i(c))$
  (and similarly for limits).
\end{proposition}
The category $C$ can be mapped into the category $\ps C$ through the Yoneda
functor, which happens to be fully faithful:
\begin{definition}
  The \emph{Yoneda functor} is the functor $\yoneda_C \co C \to \ps C$ defined
  by mapping $c \in \catob C$ to the presheaf $C(-,c) \co \cop C \to \Set$.
\end{definition}
Given $c \in \Ob(C)$ and $X \in \Ob(\ps C)$, a morphism $f \co \yoneda_C(c) \to
X$ can be mapped to an element of $X(c)$, by looking at the image of $\unit c$
by $f_c$. The family of these mappings for $c \in \Ob(C)$ defines a natural
isomorphism, as stated the Yoneda lemma:
\begin{proposition}[Yoneda lemma]
  The mapping just described induces a bijection
  \[
    \ps C(\yoneda_C(c),X) \qp\cong X(c) \qqp\in \Set
  \]
  which is natural in $c \in \Ob(C)$.
\end{proposition}


\subsection{\Lfp categories as orthogonal sub-categories}

While every presheaf category is an \lfp category, the converse is not true.
However, by a known characterisation~\cite{adamek1994locally}, every \lfp
category can be identified as a full sub-category of a presheaf category $\ps
C$. The presheaves that belong to this sub-category are the ones that are
\emph{orthogonal} to a set of \emph{orthogonality conditions}, the latter being
morphisms of $\ps C$. These conditions will enable additional structures and
equations on these orthogonal presheaves, making them expressive enough to model
objects of other \lfp categories. Since this notion of orthogonality is not at
all specific to presheaf categories, we introduce in a more general setting.
\begin{definition}[{\cite[Definition 1.32]{adamek1994locally}}]
  Let $\cC$ be a category. Let $X$ be an object of $\cC$ and $g \co A \to B$ be
  a morphism of $\cC$. We say that $X$ is \emph{orthogonal} to $g$, denoted $X
  \ortho g$, when we have the following unique lifting property: for all $f \co
  A \to X$, there exists a unique $h \co B \to X$ such that $f = h \circ g$.
  Graphically,
  \begin{equation}
    \label{eq:orthogonal-prop}
    \begin{tikzcd}[cramped,rsbo=2.5em]
      A
      \ar[rr,"g"]
      \ar[rd,"\forall f"']
      &&
      B
      \ar[ld,dashed,"\exists ! h"]
      \\
      &
      X
    \end{tikzcd}
    \zbox.
  \end{equation}
  Given a set $O$ of morphisms of $\cC$, we write $\orthocat{(\cC,O)}$, or even
  $\orthocat O$, for the full subcategory of $\cC$ whose objects are the $X \in
  \cC$ which are orthogonal to every $g \in O$.
\end{definition}
We can already mention the adjunction between $\cC$ and its orthogonal sub-category,
and whose construction will play an important role for showing the correctness
of the criterion of \Cref{sec:la-through-modeling}:
\begin{theorem}[{\cite[Construction 1.37]{adamek1994locally}}]
  \label{thm:ortho-adj}
  Given a cocomplete category $\cC$ and a set $O$ of morphisms $g \co A \to B$
  of $\cC$ such that $A$ and $B$ are finitely presentable, the category
  $\orthocat O$ is a reflective subcategory of $\cC$, that is, the embedding
  functor $\orthocatemb^{(\cC,O)} \co \orthocat O \to \cC$ (abbreviated
  $\orthocatemb$) is part of an adjunction
  \[
    \begin{tikzcd}[column sep=large]
      \orthocat{O}
      \ar[r,shift right=0.9em,"\orthocatemb"']
      \ar[r,phantom,"\perp"]
      &
      \mcal C
      \ar[l,shift right=0.9em,"\orthorefl"']
    \end{tikzcd}
  \]
  for some \emph{reflection functor} $\orthorefl^{(\cC,O)}$, for some unit
  $\eta^{(\cC,O)} \co \catunit{\cC} \To \orthocatemb\orthorefl$ and some counit
  $\eps^{(\cC,O)} \co \orthorefl\orthocatemb \To \catunit{\orthocat O}$ (often
  abbreviated $\orthorefl$, $\eta$ and $\eps$ as well).
\end{theorem}
\begin{remark}\label{rem:orthocounit-isom}
  In particular, since $\orthocatemb$ is fully faithful ($\orthocat O$ is a
  full subcategory of $\cC$), the counit $\orthocounit\co \orthorefl\orthocatemb
  \To \catunit{\orthocat O}$ is an isomorphism.
\end{remark}
\Lfp categories are then characterised the following way using orthogonality:
\begin{theorem}[{\cite[Theorem 1.46]{adamek1994locally}}]
  \label{def:lfp-categories-as-ps-ortho-cats}
  Let $\cC$ be a category. $\cC$ is \lfp iff $\cC$ is equivalent to a category
  $\orthocat{(\ps C, O)}$ for some small category $C$ and some set $O$ of
  morphisms $g \co A \to B \in \ps C$ such that both $A$ and $B$ are finitely
  presentable.
\end{theorem}
The above property justifies the following definitions:
\begin{definition}
  A \emph{presheaf model} is a pair $(C,O)$ where $C$ is a small category and
  $O$ is a set of morphisms $g \co A \to B$ of $\ps C$ such that both $A$ and
  $B$ are finitely presentable presheaves.
  A (necessarily \lfp) category $\cC$ is \emph{modeled} by the presheaf model
  $(C,O)$ when $\cC \simeq \orthocat{(C,O)}$.
\end{definition}

\subsection{Examples}

Let's now consider some examples of \lfp categories in the sense of
\Cref{def:lfp-categories-as-ps-ortho-cats}.
\begin{example}
  \label{ex:pm-set}
  The category $\Set$ is modeled by the presheaf model $(\psmcatset,\psmoset) =
  (\termcat,\emptyset)$ where $\termcat$ is the terminal category, made of one
  object $\ast \in \termcat$ and no non-trivial morphisms.
\end{example}

\begin{example}
  \label{ex:pm-set-times-set}
  While $\Set \times \Set$ can of course be modeled by the presheaf model
  $(\termcat\coprod\termcat,\emptyset)$, it can be modeled using a different and
  apparently contrived presheaf model $(\psmcatsetset,\psmosetset)$, which still
  might be useful in practice (see \Cref{ex:fm-product}). The category
  $\psmcatsetset$ is the category whose objects and generating morphisms are the
  ones of the diagram
  \[
    \begin{tikzcd}[sep=small]
      & \cstp &
      \\
      \cst{s_l}
      \ar[ru,"\catl"]
      & &
      \cst{s_r}
      \ar[lu,"\catr"']
    \end{tikzcd}
  \]
  satisfying no additional equations. The set $\psmosetset$ contains exactly one morphism
  $g^{\cstp} \co A \to B$ where $A$ is a presheaf with one element $\ast_l \in
  A(\cst{s_l})$ and one element $\ast_r \in A(\cst{s_r})$, $B$ is the presheaf
  generated by one element $\ast_p$, and $g^{\cstp}$ is the morphism sending $\ast_l$ to
  $B(\catl)(\ast_p)$ and $\ast_r$ to $B(\catr)(\ast_p)$:
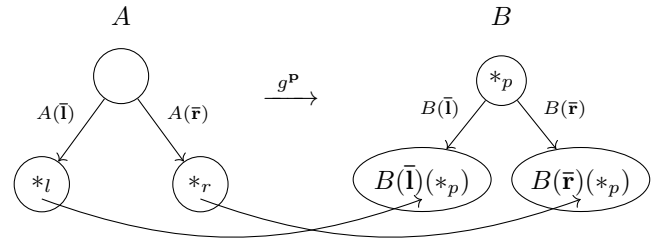
  \[
    \begin{tikzcd}[csbo=3em,rsbo=4em,remember picture]
      &A&
      \\[-18pt]
      & |[circle,draw]| \mathmakebox[3pt][c]{\vphantom{\set\ast}}
      \ar[dl,"A(\catl)"']
      \ar[dr,"A(\catr)"]
      &
      \\
      |[circle,draw,alias=diagastl]|\mytikzmark{Aastl}{\mathmakebox[3pt][c]{\ast_l}}
      & &
      |[circle,draw,alias=diagastr]|\mytikzmark{Aastr}{\mathmakebox[3pt][c]{\ast_r}}
    \end{tikzcd}
    \quad
    \xto{\mathmakebox[15pt]{g^{\cstp}}}
    \quad
    \begin{tikzcd}[csbo=3em,rsbo=4em,remember picture]
      &B&
      \\[-18pt]
      & |[circle,draw]| \mathmakebox[3pt][c]{\ast_p}
      \ar[dl,"B(\catl)"']
      \ar[dr,"B(\catr)"]
      &
      \\
      |[shape=ellipse,draw]|\mytikzmark{Bcatlastp}{\mathmakebox[20pt][c]{\vphantom{l}\smash{B(\catl)(\ast_p)}}}
      & &
      |[shape=ellipse,draw]|\mytikzmark{Bcatrastp}{\mathmakebox[20pt][c]{\vphantom{l}\smash{B(\catr)(\ast_p)}}}
      \ar[from=Aastl,start anchor={[yshift=-4pt]base},to=Bcatlastp,end anchor={[yshift=-4pt]base},bend right=20,overlay]
      \ar[from=Aastr,start anchor={[yshift=-4pt]base},to=Bcatrastp,end anchor={[yshift=-4pt]base},bend right=20,overlay]
    \end{tikzcd}
  \]
  \vskip12pt
  An object $X \in \pswide \psmcatsetset$ which is orthogonal to $g^{\cstp}$ is then an object such that
  $(X(\cstp),X(\cstl),X(\cstr))$ is a product of $X(\cst{s_l})$ and
  $X(\cst{s_r})$. Indeed, when $X \ortho g^{\cstp}$, we have
  \[
    X(\cstp)
    \cong
    \pswide \psmcatsetset(B,X)
    \cong
    \pswide \psmcatsetset(A,X)
    \cong
    X(\cst{s_l}) \times X(\cst{s_r})\zbox.
  \]
  Thus, $X$ is essentially given by $X(\cst{s_l})$ and $X(\cst{s_r})$. Hence,
  $(\psmcatsetset,\psmosetset)^{\ortho} \simeq \Set\times \Set$.
\end{example}

\begin{example}
  \label{ex:pm-cat}
  The category $\Cat$ of small categories and functors can be modeled by a
  presheaf model $(\psmcatcat,\psmocat)$ defined as follows: $\psmcatcat$ is the
  category with objects and whose generating morphisms given by the diagram
  \[
    \begin{tikzcd}
      \csto
      \ar[r,shift right=7pt,"\catsrc"']
      \ar[r,shift left=7pt,"\cattgt"]
      &
      \cstm
      \ar[l,"{\catid}"{description}]
      \ar[r,"\catcomp"{description}]
      \ar[r,shift right=7pt,"\catr"']
      \ar[r,shift left=7pt,"\catl"]
      &
      \cstp
    \end{tikzcd}
  \]
  where $\csto$ and $\cstm$ represent the sets of objects and morphisms, and
  where $\cstp$ represents the set of \emph{composable} morphisms (that is,
  morphisms $u,v$ such that the target of $u$ is the source of $v$), with
  $\cstl$ and $\cstr$ as left and right projections, and $\cstcomp$ as
  composition operation. The morphisms of $\psmcatcat$ are moreover required to
  satisfy the following equations: $\catid \circ \catsrc = \unit{\csto}$ and
  $\catid \circ \cattgt = \unit{\csto}$ (source and target of an identity on an
  object is the object itself), $\catcomp \circ \catsrc = \catl\circ \catsrc$
  and $\catcomp \circ \cattgt = \catr\circ \cattgt$ (the source/target of a
  composition is the source/target of the first/second argument); and $\catl
  \circ \cattgt = \catr \circ \catsrc$ (the target of the left projection of an
  element of $\cstp$ is the source of the right projection).

  Now a presheaf on $\psmcatcat$ is still not a small category: we still need to enforce
  some additional conditions through $\psmocat$. First, we need to force $\cstp$ to
  correspond to pairs of composable morphisms of $\cstm$. This is done through
  orthogonality against the following $g^\cstp \co A^\cstp \to B^\cstp$ defined
  as in the following diagram
  
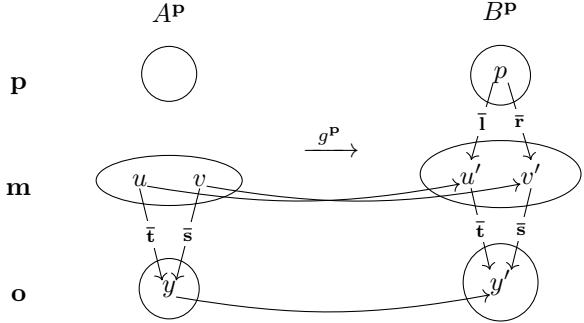
\[
    \begin{tikzcd}[csbo=3em,rsbo=4em,remember picture]
      \vphantom{A}
      \\[-18pt]
      \cstp
      \\
      \cstm
      \\
      \csto
    \end{tikzcd}
    \qquad
    \begin{tikzcd}[csbo=3em,rsbo=4em,remember picture]
      A^\cstp
      \\[-18pt]
      |[circle,draw]| \mathmakebox[3pt][c]{\vphantom{\set\ast}}
      \\
      |[ellipse,draw]| \mytikzmark{catpairAu}{u}\hspace*{1.5em}\mytikzmark{catpairAv}{v}
      \\
      |[ellipse,draw]| \mytikzmark{catpairAy}{y}
    \end{tikzcd}
    \qquad
    \xto{\mathmakebox[15pt]{g^\cstp}}
    \qquad
    \begin{tikzcd}[csbo=3em,rsbo=4em,remember picture]
      B^\cstp
      \\[-18pt]
      |[circle,draw]| \mytikzmark{catpairBp}{p}
      \\
      |[ellipse,draw]| \mytikzmark{catpairBup}{u'}\hspace*{1.1em}\mytikzmark{catpairBvp}{v'}
      \\
      |[ellipse,draw]| \mytikzmark{catpairByp}{y'}
      \ar[from=catpairAu,start anchor=south east,to=catpairBup,end anchor=south
      west,bend right=10,overlay]
      \ar[from=catpairAv,start anchor=south east,to=catpairBvp,end anchor=south
      west,bend right=10,overlay]
      \ar[from=catpairAy,start anchor=south east,to=catpairByp,end anchor=south
      west,bend right=10,overlay]
      \ar[from=catpairAu,start anchor=south,to=catpairAy,end anchor=north west,overlay,"\cattgt"{description}]
      \ar[from=catpairAv,start anchor=south,to=catpairAy,end anchor=north east,overlay,"\catsrc"{description}]
      \ar[from=catpairBup,start anchor=south,to=catpairByp,end anchor=north west,overlay,"\cattgt"{description}]
      \ar[from=catpairBvp,start anchor=south,to=catpairByp,end anchor=north east,overlay,"\catsrc"{description}]
      \ar[from=catpairBp,start anchor=south west,to=catpairBup,end anchor=north,overlay,"\catl"{description}]
      \ar[from=catpairBp,start anchor=south east,to=catpairBvp,end anchor=north,overlay,"\catr"{description}]
    \end{tikzcd}
  \]
  Here, we simplified the description of $A^\cstp$, $B^\cstp$ and $g^\cstp$
  compared to the way we did in \Cref{ex:pm-set-times-set}. Indeed, we only
  showed some set of generating elements ($y,u,v$ for $A^\cstp$, $y',u',v',p$
  for $B^\cstp$) together with some images by some morphisms of $\psmcatcat$,
  using a light notation. For example, the arrow $u \xto{\cattgt} y$ means that
  $A^\cstp(\cattgt)(u) = y$. The presheaves $A^\cstp$ and $B^\cstp$ are then
  defined as the \eq{free presheaves} satisfying the data of the diagram. For
  example, $A^\cstp$ is the presheaf freely generated by two elements $u,v \in
  A^\cstp(\cstm)$ and one element $y \in A^\cstp(\csto)$, enforcing moreover
  that $A^\cstp(\cattgt)(u) = y = A^\cstp(\catsrc)(v)$.

  Given $X \in \pswide \psmcatcat$, the fact that $X \ortho g^\cstp$ means exactly that $X(\cstp)$
  is the set of pairs of composable arrows $u,v \in X(\cstm)$, with $X(\catl)$
  and $X(\catr)$ the left and right projections.

  We still need to enforce the axioms of categories, namely left and right
  unitality, and associativity through morphisms
  $g^{\cstlu},g^{\cstru},g^{\cstass}$. We only show the definition of
  $g^{\cstlu}$, since the definition of the other morphisms follow the same
  ideas. The morphism $g^\cstlu \co A^\cstlu \to B^\cstlu$ is thus defined as in
  the following diagram
  
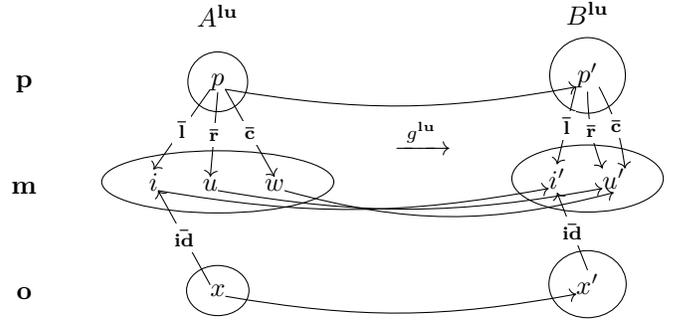
\[
    \begin{tikzcd}[csbo=3em,rsbo=4em,remember picture]
      \vphantom{A}
      \\[-18pt]
      \cstp
      \\
      \cstm
      \\
      \csto
    \end{tikzcd}
    \qquad
    \begin{tikzcd}[csbo=3em,rsbo=4em,remember picture]
      A^\cstlu
      \\[-18pt]
      |[circle,draw]| \mytikzmark{catluAp}{p}
      \\
      |[ellipse,draw]| \mytikzmark{catluAi}{i}\hspace*{1.5em}\mytikzmark{catluAu}{u}\hspace*{1.5em}\mytikzmark{catluAw}{w}
      \\
      |[ellipse,draw]| \mytikzmark{catluAx}{x}
    \end{tikzcd}
    \qquad
    \xto{\mathmakebox[15pt]{g^\cstlu}}
    \qquad
    \begin{tikzcd}[csbo=3em,rsbo=4em,remember picture]
      B^\cstlu
      \\[-18pt]
      |[circle,draw]|\mytikzmark{catluBpp}{p'}
      \\
      |[ellipse,draw]| \mytikzmark{catluBip}{i'}\hspace*{1.1em}\mytikzmark{catluBup}{u'}
      \\
      |[ellipse,draw]| \mytikzmark{catluBxp}{x'}
      \ar[from=catluAp,start anchor=south east,to=catluBpp,end anchor=south west,bend right=10,overlay]
      \ar[from=catluAi,start anchor=south east,to=catluBip,end anchor=south
      west,bend right=10,overlay]
      \ar[from=catluAu,start anchor=south east,to=catluBup,end anchor=south
      west,bend right=10,overlay]
      \ar[from=catluAw,start anchor=south east,to=catluBup,end anchor=south,bend right=15,overlay]
      \ar[from=catluAx,start anchor=south east,to=catluBxp,end anchor=south
      west,bend right=10,overlay]
      \ar[from=catluAp,start anchor=south west,to=catluAi,end anchor=north,overlay,"\catl"{description}]
      \ar[from=catluAp,start anchor=south,to=catluAu,end anchor=north,overlay,"\catr"{description}]
      \ar[from=catluAp,start anchor=south east,to=catluAw,end anchor=north,overlay,"\catcomp"{description}]
      \ar[from=catluAx,start anchor=north west,to=catluAi,end anchor=south east,overlay,"\catid"{description}]
      \ar[from=catluBpp,start anchor=south west,to=catluBip,end anchor=north,overlay,"\catl"{description}]
      \ar[from=catluBpp,start anchor=south,to=catluBup,end anchor=north
      west,bend right=10,overlay,"\catr"{description}]
      \ar[from=catluBpp,start anchor=south east,to=catluBup,end anchor=north
      east,bend left=10,overlay,"\catcomp"{description}]
      \ar[from=catluBxp,start anchor=north,to=catluBip,end anchor=south,overlay,"\catid"{description}]
    \end{tikzcd}
  \]
  Given $X \in \pswide \psmcatcat$, the fact that $X \ortho g^\cstlu$ means exactly that
  given $p \in X(\cstp)$ such that the left projection of $p$ is an identity,
  that it, $X(\catl)(p) = X(\catid)(x)$ for some $x \in X(\cato)$, then the
  composition $X(\catcomp)(p)$ of $p$ is just its right projection
  $X(\catr)(p)$, which is essentially left unitality.

  Putting $\psmocat = \set{g^\cstp,g^\cstlu,g^\cstru,g^\cstass}$, we easily
  verify that $X \in \pswide \psmcatcat$ such that $X \ortho g$ for $g \in
  \psmocat$ is essentially the data of a small category. So that
  $(\psmcatcat,\psmocat)$ is a presheaf model of $\Cat$.
\end{example}

\begin{example}
  Following what we have done for $\Cat$ in the previous example, it is
  similarly possible to model categories of algebraic structures (categories
  $\Grp$ of groups, $\Rng$ of rings, \etc) using presheaf models.
\end{example}

\section{Modelling functors between categories}
\label{sec:modeling-functors}

Now that we have a convenient description for an interesting class of
categories, we now introduce a similar description of the functors between such
categories.
For this purpose, we propose another modeling device, called a \emph{Kan model},
which is able to model a decent class of such functors and which can be encoded
as well under reasonable hypotheses. The modeled functors are then recovered
from these devices using \emph{left Kan extensions}, the definition of which
relies on the notions of \emph{tensors} and \emph{coends} that we recall below.

\subsection{Tensors}

The notion of tensor discussed here is an instance of a more general notion of
tensors for enriched categories~\cite{kelly1982basic,borceux1994handbook2}. In
our unenriched setting, tensors will be quite easy to define, but will still be
useful in our developments.
\begin{definition}
  \label{def:tensor}
  Given a category with coproducts $\cC$, an object $X \in \cC$ and $S \in
  \Set$, we write $X \otimes S$ for the \emph{tensor} of $X$ by $S$, defined as
  \[
    X \otimes S \qp= \coprod_{s \in S} X
    \zbox.
  \]
  This operation naturally extends to a functor
  \[
    (-)_1 \otimes (-)_2
    \co
    \quad
    \cC \times \Set
    \qp\to
    \cC
    \zbox.
  \]
\end{definition}
Tensors satisfy the following universal property, which is normally expressed in
the broader setting of enriched categories:
\begin{proposition}
  Given a category $\cC$ as in \Cref{def:tensor}, and $X,Y \in \cC$ and $S \in
  \Set$, the copairing operations on morphisms induces a natural isomorphism
  \[
    \Hom_{\Set}(S,\Hom_{\cC}(X,Y)) \cong \Hom_{\cC}(X \otimes S, Y)\zbox.
  \]
\end{proposition}
As a direct consequence, we have:
\begin{proposition}
  \label{prop:otimes-pres-colimits}
  Given $X \in \cC$ and $S \in \Set$, the functors $X \otimes (-) \co \Set \to
  \cC$ and $(-) \otimes S \co \cC \to \cC$ preserve colimits.
\end{proposition}

\subsection{Coends}

Another tool that will be useful to us is the notion of coends that we recall
briefly below in a restricted case. We refer the reader to the existing
literature, like \cite[Sec.~IX.6]{mac2013categories} or~\cite{Loregian2021} for
more details:

\begin{definition}
  Given two categories $C$ and $\cD$ and a functor $H \co C \times \catop C
  \to \cD$, a \emph{cowedge} for $H$ is the data of an object $W$ of $\cD$ and
  morphisms $w_c \co H(c,c) \to W$ for each $c \in C$ such that the
  squares
  \begin{equation}
    \label{eq:coend-def-squares}
    \begin{tikzcd}[cramped,rsbo=2.5em]
      H(c,c')
      \ar[r,"{H(f,c')}"]
      \ar[d,"{H(c,f)}"']
      &
      H(c',c')
      \ar[d,"{w_{c'}}"]
      \\
      H(c,c)
      \ar[r,"{w_{c}}"']
      &
      W
    \end{tikzcd}
  \end{equation}
  commutes. A \emph{coend}, denoted $\coend^{c\in C}H(c,c)$ is a cowedge which
  is initial among cowedges, that is, such that any other cowedge $W$ on $H$
  induces a unique morphism of cowedges $\coend^{c\in C}H(c,c) \to W$.
\end{definition}
\begin{remark}
  \label{rem:coend-as-quotient}
  In the case where $C$ is small and $\cD$ is cocomplete, $\coend^{c\in
    C}H(c,c)$ can be alternatively characterised (see \cite[Proposition
  IX.5.1]{mac2013categories}) as the quotient
  \[
    \begin{tikzcd}[csbo=5em]
      \coprod\limits_{\mathmakebox[18pt][c]{f \co c \to c' \in C}} (H(c,c'))
      \ar[r,shift left,"{[\iota_{c'} \circ (H(f,c'))]_{f : c \to c' \in C}}"]
      \ar[r,shift right,"{[\iota_{c} \circ (H(c,f))]_{f : c \to c' \in C}}"']
      &[8em]
      \coprod\limits_{c \in C} H(c,c)
      \ar[r,dashed]
      &[1.8em]
      \coend^{c\in C}H(c,c)
    \end{tikzcd}
  \]
  of the coproduct $\coprod_{c \in C} H(c,c)$ under the conditions given by the
  squares~\eqref{eq:coend-def-squares}, where we wrote $\iota_c$ for the
  coprojections and $[-]_{f \in C}$ for the copairing operation (we will use
  again these notations in the following).%
  \simon{ça ne sert à rien
    d'expliciter ce quotient dans la mesure où on ne l'utilisera pas tel quel,
    car $C$ sera finiment générée.}%
\end{remark}
The notion of coends allows us to express an important proposition for
presheaves. Given a small category $C$ and $X \in \ps C$, by the Yoneda lemma,
we have a canonical morphism $\yoneda_C(c) \otimes X(c) \to X$ for every $c \in
\Ob(C)$. These morphisms even factor through the coend as a morphism $\int^{c
  \in C} \yoneda_C(c) \otimes X(c) \to X$ which is actually an isomorphism:
\begin{proposition}[Density formula/co-Yoneda lemma]
  \label{prop:density-formula}
  Given a small category $C$ and $X \in \ps C$, the canonical morphism
  \[
    \int^{c \in C} \yoneda_C(c) \otimes X(c) \to X
  \]
  is an isomorphism.
\end{proposition}
\begin{proof}
  See for example the proof of \cite[Proposition 2.2.1]{Loregian2021}.
\end{proof}

\subsection{Kan extensions}

Here we recall a working definition of a Kan
extension~\cite{mac2013categories,borceux1994handbook2} in the restricted case
that will interest us, that is, when the functor along which a Kan extension is
taken is the Yoneda embedding. Kan extensions will be the key construction to
present functors using Kan models in the next section.

\begin{definition}
  \label{def:kan-extension}
  Given a small category $C$, a cocomplete category $\cD$ and a functor $F\co C
  \to \cD$ as in the above definition, the \emph{left Kan extension} of $F$ (along
  $\yoneda_C$) is a functor $\Lan F \co \ps C \to \cD$, defined for $X \in
  \catob{\ps C}$ by
  \[
    \Lan F(X) = \int^{c \in C} F(c) \otimes X(c)
  \]
  and naturally extended to morphisms of presheaves $f \co X \to X'$.
\end{definition}
As can be intuited from unfolding the definitions, given $X \in \catob{\ps C}$,
$\Lan F(X)$ is an object constructed with one copy of $F(c)$ for every $x \in
X(c)$ for each $c \in \catob{C}$, adequately glued together by the coend
operator: in other words, an \emph{$X$-weighted colimit} of $F$ (see \cite[Section
6.6]{borceux1994handbook2}).
\begin{remark}
  The left Kan extension
  \[
    \Lan(-) \co \CAT(C,\cD) \to \CAT(\ps C,\cD)
  \]
  provides a left adjoint to the precomposition functor
  \[
    (-) \circ \yoneda_C \co \CAT(\ps C,\cD) \to \CAT(C,\cD)\zbox.
  \]
  The density formula (\Cref{prop:density-formula}) then states that the counit
  of this adjunction is an isomorphism, as a consequence of the \emph{density}
  (see \cite[Section X.6]{mac2013categories}) of $\yoneda_C$.
\end{remark}

\subsection{Representation for functors}

We can now introduce our description of functors between two categories that are
modeled by presheaf models. It will rely on the following device:
\begin{definition}
  Given two presheaf models $(C,O)$ and $(D,P)$ of two categories $\cC$ and
  $\cD$ respectively, a \emph{Kan model} is a functor $F \co C \to \ps D$.
\end{definition}
\begin{remark}
  Technically, Kan models are just \emph{profunctors}, also called
  \emph{distributors}~\cite{borceux1994handbook}. But since the more general
  theory of profunctors does not seem to be that much relevant to our purposes,
  we preferred to use an independent name for their use in this work.
\end{remark}

From a Kan model $F \co C \to \ps D$ as in the above definition, we now
introduce in the following definition the steps to recover a functor $\bar F \co
\cC \to \cD$:
\begin{definition}
  Let $\cC$ and $\cD$ be \lfp categories and $(C,O)$ and $(D,P)$ be associated
  presheaf models. Given a Kan model $F \co C \to \ps D$ we define functors $F'
  \co \ps C \to \ps D$, $\tilde F \co \ps C \to \orthocat P$, $\tilde F' \co
  \orthocat O \to \orthocat P$ and $\bar F \co \cC \to \cD$ as follows. We first
  define the functor $F' \co \ps C \to \ps D$ by $F' = \Lan F$. Then, we build
  the functor $\tilde F$ by putting
  \[
    \tilde F
    \qeq
    \begin{tikzcd}[csbo=5em]
      \ps C
      \ar[r,"F'"]
      &
      \ps D
      \ar[r,"\orthorefl"]
      &
      \orthocat P
      \zbox.
    \end{tikzcd}
  \]
  Next, we define the functor $\tilde F' \co \orthocat O \to \orthocat P$ by putting
  \[
    \tilde F'
    \qeq
    \begin{tikzcd}[csbo=5em]
      \orthocat O
      \ar[r,"\orthocatemb"]
      &
      \ps C
      \ar[r,"\tilde F"]
      &
      \orthocat P
      \zbox.
    \end{tikzcd}
  \]
  Finally, we define $\bar F \co \cC \to \cD$ as the composition
  \[
    \bar F
    \qeq
    \begin{tikzcd}[csbo=5em]
      \cC
      \ar[r,"\simeq"]
      &
      \orthocat O
      \ar[r,"\tilde F'"]
      &
      \orthocat P
      \ar[r,"\simeq"]
      &
      \cD
      \zbox.
    \end{tikzcd}
  \]
  A functor $\cF \co \cC \to \cD$ is \emph{modeled} by $F$ when $\cF \cong \bar
  F$.
\end{definition}

\begin{example}
  \label{ex:ob-functor-model}
  Relatively to the presheaf models of \Cref{ex:pm-cat,ex:pm-set}, the object
  functor $\catob- \co \Cat \to \Set$ mapping a small category $D$ to its set of
  objects $\catob D$ can be modeled by the Kan model $\FOb \co \psmcatcat \to
  \pswide\psmcatset$ defined as the co-Yoneda embedding $\psmcatcat(\csto,-) \co
  \psmcatcat\to \Set$ (for simplicity, we directly identify $\pswide\psmcatset$
  with $\Set$). Indeed, by the density formula (\Cref{prop:density-formula}), we
  have that, for $X \in \pswide\psmcatcat$, $X(-) \cong \coend^{c \in
    \psmcatcat} \yoneda_{\psmcatcat}(c) \otimes X(c)$. In particular, $X(\csto) \cong
  \coend^{c \in \psmcatcat} \psmcatcat(\csto,c) \otimes X(c)$. Of course, this
  formula stays true when we consider $X \in \orthocat{(\psmcatcat,\psmocat)}$.
  Moreover, in the correspondence $\Cat \simeq
  \orthocat{(\psmcatcat,\psmocat)}$, a small category $D$ is sent to some $X$
  such that $\catob D \cong X(\csto)$. Thus, $\FOb$ is modelling $\Ob$. We
  compute that, up to the identification of $\pswide\psmcatset$ with $\Set$,
  $\FOb$ is a presheaf such that $\FOb(\csto)$ has one element, $\FOb(\cstm)$
  has two elements and $\FOb(\cstp)$ has three elements.
\end{example}

\begin{example}
  \label{ex:fm-product}
  The product functor of sets $(X,Y) \mapsto X \times Y$ cannot be modeled
  through a Kan model between the presheaf model of
  $(\termcat\coprod\termcat,\emptyset)$ to $(\psmcatset,\psmoset)$. Indeed,
  given $F \co \termcat\coprod\termcat \to \pswide\psmcatset$, we have that $F'
  = \Lan F \co \pswide{\termcat\coprod\termcat}\to \pswide\psmcatset$ satisfies
  \[
    \Lan F(Z)
    \qp\cong
    (\coprod_{\mathmakebox[0pt][c]{i \in Z(\coproj_l(\ast))}} F(\coproj_l(\ast)))
    \quad
    \coprod
    \quad (\coprod_{\mathmakebox[0pt][c]{i \in Z(\coproj_r(\ast))}} F(\coproj_r(\ast)))
  \]
  for $Z \in\pswide{\termcat\coprod\termcat}$. In particular, it is impossible
  to have $\Lan F(Z) \cong Z(\coproj_l(\ast)) \times Z(\coproj_r(\ast))$ for every
  $Z$.

  However, the product functor can be modeled using the presheaf model
  $(\psmcatsetset,\psmosetset)$ for $\Set \times \Set$ of
  \Cref{ex:pm-set-times-set}. Indeed, by the density formula
  (\Cref{prop:density-formula}), we have that, for $X \in \pswide\psmcatsetset$,
  $X(-) \cong \coend^{c \in \psmcatsetset} \yoneda_{\psmcatsetset}(c) \otimes
  X(c)$. In particular, we have $X(\cstp) \cong \coend^{c \in \psmcatsetset}
  \psmcatsetset(\cstp,c) \otimes X(c)$. Of course, this formula stays true when
  we consider $X \in \orthocat{(\psmcatsetset,\psmosetset)}$, for which
  $X(\cstp) \cong X(\cst{s_l}) \times X(\cst{s_r})$. Thus, taking $\Fprod \co
  \psmcatsetset \to \pswide\psmcatset \simeq \Set$ to be
  $\psmcatsetset(\cstp,-)$, we have that $\Fprod$ models the product functor.
\end{example}
We will derive an important example from the following proposition:
\begin{proposition}
  \label{prop:density-for-product}
  Given a presheaf model $(C,O)$ and $A,B \in \orthocat{(\ps C,O)}$, we have
  that
  \[
    \begin{aligned}
      A \times B
      &
      \qp\cong
      \orthorefl(\coend^{\mathmakebox[5pt][l]{c \in C}} (\yoneda_{C}(c) \times \orthocatemb B) \otimes (\orthocatemb A)(c))
        \\
      &
      \qp\cong
      \coend^{\mathmakebox[5pt][l]{c \in C}} \orthorefl(\yoneda_{C}(c) \times \orthocatemb B) \otimes (\orthocatemb A)(c)
      \zbox.
    \end{aligned}
  \]
\end{proposition}
\begin{proof}
  By the density formula (\Cref{prop:density-formula}), and since the functor
  $(-) \times \orthocatemb B$ commutes with colimits, as a consequence of the
  cartesian closure of $\ps C$ (\Cref{prop:pscat-props,prop:charact-closure}),
  we have that
  \[
    \begin{aligned}
      \orthocatemb A \times \orthocatemb B
      &\cong
      (\coend^{c \in C} (\yoneda_{C}(c) \otimes (\orthocatemb A)(c))) \times \orthocatemb B
      \\
      &\cong
      \coend^{c \in C} (\yoneda_{C}(c) \times \orthocatemb B) \otimes (\orthocatemb A)(c)
      \zbox.
    \end{aligned}
  \]
  Moreover, since $\orthorefl$ is a left adjoint, it commutes with colimits, so that
  \[
    \begin{aligned}
      &\orthorefl(\coend^{c \in C} (\yoneda_{C}(c) \times \orthocatemb B) \otimes (\orthocatemb A)(c))
        \\
      \cong&
      \coend^{c \in C} \orthorefl(\yoneda_{C}(c) \times \orthocatemb B) \otimes (\orthocatemb A)(c)\zbox.
    \end{aligned}
  \]
  We conclude, using \Cref{rem:orthocounit-isom}, by observing that
  \[
    A \times B \cong \orthorefl \orthocatemb (A \times B) \cong \orthorefl
    (\orthocatemb A \times \orthocatemb B)
    \zbox.
    \qedhere
  \]
\end{proof}
\begin{example}
  \label{ex:times-functor-model}
  Let $\cC$ be an \lfp category modeled by a presheaf model $(C,O)$. Given $B
  \in \cC$, by \Cref{prop:density-for-product}, the functor $- \times B \co \cC
  \to \cC$ can be modeled by the Kan model $\Ftimes \co C \to \ps C$ defined by
  $\Ftimes(-) = \yoneda_{C}(-)\times \orthocatemb B$.
\end{example}

\section{Left adjointness through modeling}
\label{sec:la-through-modeling}

We now apply our modeling of categories and functors to tackle the problem of
showing that a functor is a left adjoint. For this purpose, we introduce an
effective criterion that relies on this modeling. As we will show in a later
section, under reasonable finiteness hypotheses, this criterion can be
computationally checked, either fully automatically in good cases, or by
providing proof assistance to the user.

\subsection{The reflection construction}

Our criterion will be derived from a careful study of the proof of the
adjunction between an orthogonal subcategory and its containing category, as is
stated by \Cref{thm:ortho-adj}. Although we will not recall the whole proof of
that theorem here, we will at least describe the construction of the functor
$\orthorefl$ on objects. For that, it is useful to analyse how the object $X$
can fail to be in $\orthocat O$ and what can be done to fix that. Consider $f
\co X \to A$ and $g \co A \to B \in O$. First, there can be no $h \co B \to X$
such that \eqref{eq:orthogonal-prop} commutes. In this case, we can \eq{change}
$X$ for such a factorisation to exist: we take the pushout
\[
  \begin{tikzcd}[cramped,rsbo=2.5em]
    A
    \ar[r,"g"]
    \ar[d,"f"']
    &
    B
    \ar[d,"p",dashed]
    \\
    X
    \ar[r,"q"',dashed]
    &
    X'
  \end{tikzcd}
\]
so that now, $q \circ f \co A \to X'$ admits a factorisation through $g$.
Second, it can be the case that there are two different factorisations $h,h'
\co B \to X$ for \eqref{eq:orthogonal-prop}. In this case, we \eq{change}
again $X$ for these two factorisations to be identified: we take the coequaliser
\[
  \begin{tikzcd}
    B
    \ar[r,shift left,"h"]
    \ar[r,shift right,"h'"']
    &
    X
    \ar[r,dashed,"r"]
    &
    X'
  \end{tikzcd}
\]
so that the factorisations $r \circ h$ and $r \circ h'$ of $r \circ f$ are now
identified. Of course, after applying one \eq{change} to $X$, several other
factorisations might still be missing or non-unique, and some new
factorisation problems might arise on $X'$. In order to construct the image of
$X$ by the functor $\orthorefl$, we use an iterative procedure where, at each step, we
try to be exhaustive on the solving of the existing factorisation problems.

Given an object $X\in \cC$, the object $LX$ is defined as the result (or, more
precisely, the colimit) of an iterative construction
\[
  \begin{tikzcd}[csbo=3em]
    X_0
    \ar[rr]
    &&
    X_1
    \ar[rr]
    &&
    X_2
    \ar[rr]
    &&
    \cdots
  \end{tikzcd}
\]
where $X_0 = X$ and, given $i \in \N$, $X_{i+1}$ is defined from $X_i$ as
follows. Given the canonical enumeration $(g^e_{i,j} \co A^e_{i,j} \to
B^e_{i,j},f^e_{i,j} \co B^e_{i,j}\to X_i)_{j\in J^e}$, as projections, of the
set
\[
  \begin{split}
    J^{i,e} = \{(g, f) \mid g \co A \to B\in O,\;f \co A \to X_i \in \cC,\\
    \not\exists h \text{ s.t. } h \circ g = f\}
  \end{split}
\]
and the canonical enumeration $(g^u_{i,j},h^u_{i,j},h'^u_{i,j})_{j \in J^{i,u}}$ of the set
\[
  \begin{split}
  J^{i,u} = \{(g,h,h') \mid g \co A \to B\in O,\;h,h' \co B \to X_i \in \cC,\;
    \\
    h \circ
    g = h' \circ g \text{ and } h \neq h'\}\zbox,
  \end{split}
\]
we build the diagram
\begin{equation}
  \label{eq:step-Xi-diag}
  \begin{tikzcd}[cramped]
    &
    \smash{\mathmakebox[0pt][r]{\raisebox{1.2em}{$\ddots$}}}\;A^e_{i,j}\;\smash{\mathmakebox[0pt][l]{\raisebox{-1.2em}{$\ddots$}}}
    \ar[r,"g^e_{i,j}"]
    \ar[d,"f^e_{i,j}"{description}]
    &
    B^e_{i,j}
    \\
    \smash{\mathmakebox[0pt][l]{\;\;\raisebox{1.2em}{$\vdots$}}}\smash{\mathmakebox[0pt][l]{\;\;\raisebox{-1.5em}{$\vdots$}}}B^u_{i,j}
    \ar[r,"h^u_{i,j}",shift left]
    \ar[r,"h'^u_{i,j}"',shift right]
    &
    X_i
  \end{tikzcd}
\end{equation}
whose set of objects can be indexed by $J^{i,e} \sqcup J^{i,e} \sqcup J^{i,u}
\sqcup \set{\ast}$, where $\set{\ast}$ represents the object of the diagram sent
to $X_i$. By taking the colimit on this diagram, we handle all the instances of
the two operations of \eq{change} above at once. We write $X_{i+1}$ for the
colimit of this diagram, and $\transemb_i \co X_i \to X_{i+1}$ for the
associated coprojection.

The \emph{reflection} $LX$ of an object $X \in \cC$ is then obtained as a
colimit in $\cC$
\begin{equation}
  \label{eq:reflection-def}
  \begin{tikzcd}[cramped,rsbo=2.5em,csbo=3em]
    &&&\orthocatemb \orthorefl X&&&
    \\
    X_0
    \ar[rr,"\transemb_0"{description}]
    \ar[rrru,"\eta_X",dashed]
    &&
    X_1
    \ar[rr,"\transemb_1"{description}]
    \ar[ru,dashed]
    &&
    X_2
    \ar[rr,"\transemb_2"{description}]
    \ar[lu,dashed]
    &&
    \cdots
    \ar[lllu,dashed]
  \end{tikzcd}
\end{equation}
where we write $\eta_X \co X \to \orthocatemb\orthorefl X$ for the coprojection
associated with $X_0 = X$. The above transfinite composition ensures that
$\orthocatemb\orthorefl X$ is saturated by the two operations of \eq{changes}
discussed at the beginning, so that $\orthorefl X \in \orthocat O$ is
well-defined.

\subsection{A criterion for left adjointness}

Let $\cC$ and $\cD$ be two categories modeled by presheaf models $(C,O)$ and
$(D,P)$, and $\cF \co \cC \to \cD$ be a functor modeled by a Kan model $F \co C
\to \ps D$. In this section, we give a criterion based on $F$ to determine
whether $\cF$ is a left adjoint.

First, since $\cF \cong \bar F$, it amounts to find such a criterion directly
for $\bar F$, or even $\tilde F'$, since $\cC \simeq \orthocat O$ and $\cD
\simeq \orthocat P$. We have the following folklore result for \lfp categories:
\begin{proposition}
  \label{prop:lfp-charact-la}
  A functor between two \lfp categories is a left adjoint if and only if it
  preserves all colimits.
\end{proposition}
\begin{proof}
  In order to prove that a functor $\cF \co \cC \to \cD$ between two categories
  is a left adjoint, it is enough, by a dualised version of the Special Adjoint
  Functor Theorem (see \cite[Theorem 3.3.4]{borceux1994handbook}), to show that:
  \begin{enumerate}
  \item $\cC$ is cocomplete;
  \item $\cF$ preserves small colimits;
  \item $\cC$ is co-well-powered;
  \item $\cC$ has a generating family.
  \end{enumerate}
  But, in the case where $\cC$ is \lfp, we have that $\cC$ is cocomplete (by
  definition), $\cC$ is co-well-powered (see \cite[Theorem
  1.58]{adamek1994locally}) and $\cC$ has a generating family (see \cite[Theorem
  1.20]{adamek1994locally}). Thus, the property holds.
\end{proof}
\begin{proposition}
  \label{prop:Fp-pres-colimits}
  The functor $F' \co \ps C \to \ps D$ preserves colimits.
\end{proposition}
\begin{proof}
  Let $I$ be a small category and $(X_{(-)} \co I \to \ps C)_{i \in I}$ be a
  diagram on $\ps C$. We compute that
  \begin{align}
    \label{align:fp-pres-colim:one}
    \Lan F(\colim_i X_i)
    &\cong
    \int^{c \in C} F(c) \otimes ((\colim_i X_i)(c))
      \\
    \label{align:fp-pres-colim:two}
    &\cong
      \int^{c \in C} F(c) \otimes \colim_i (X_i(c))
    \\
    \label{align:fp-pres-colim:three}
    &\cong
      \int^{c \in C} \colim_i(F(c) \otimes  X_i(c))
    \\
    \label{align:fp-pres-colim:four}
    &\cong
      \colim_i\left(  \int^{c \in C} (F(c) \otimes  X_i(c))\right)
    \\
    \label{align:fp-pres-colim:five}
    &\cong
      \colim_i \Lan F(X_i)
  \end{align}
  where the steps are justified as follows: from \eqref{align:fp-pres-colim:one}
  to~\eqref{align:fp-pres-colim:two}, because colimits in presheaves are
  computed pointwisely; from \eqref{align:fp-pres-colim:two}
  to~\eqref{align:fp-pres-colim:three}, by \Cref{prop:otimes-pres-colimits};
  from \eqref{align:fp-pres-colim:three} to~\eqref{align:fp-pres-colim:four} by
  the commutativity of colimits with colimits, since coends can be expressed as
  colimits as we saw in \Cref{rem:coend-as-quotient}.
\end{proof}

\begin{proposition}
  \label{prop:tildeF-pres-colimits}
  The functor $\tilde F \co \ps C \to \orthocat P$ preserves colimits.
\end{proposition}
\begin{proof}
  The functor $\orthorefl \co \ps D \to \orthocat P$ preserves colimits as a
  left adjoint (see \Cref{thm:ortho-adj}), so that $\tilde F$ does as well as
  the composition of two functors which preserve colimits by
  \Cref{prop:Fp-pres-colimits}.
\end{proof}

\begin{proposition}
  \label{prop:barF-pres-colim-iff-tildeFpetacolim-isom}
  Given a Kan model $F \co C \to \ps D$, the functor $\tilde F' \co \orthocat O
  \to \orthocat P$ preserves colimits if and only if $\tilde F
  (\eta_{\colim^{\ps C}_i \orthocatemb A_i})$ is an isomorphism for all small
  diagram $A_{(-)} \co I \to \orthocat O$.
\end{proposition}

\begin{proof}
  Let $A_{(-)} \co I \to \orthocat O$ be a diagram in $\orthocat O$. We first
  observe that the colimit of $A_{(-)}$ in $\orthocat O$ is computed as the
  reflection of the one in $\ps C$. Indeed, for every $i \in I$, we have a
  natural isomorphism $A_{i} \cong \orthorefl \orthocatemb A_i$ which is natural
  in $i$. Thus, $\colim_i A_i \cong \colim_i \orthorefl \orthocatemb A_i$, and
  since $\orthorefl$ preserves colimits, we moreover have $\colim_i A_i \cong
  \orthorefl (\colim_i \orthocatemb A_i)$.

  Since $\colim_i \tilde F' A_i$ is a colimit cocone, it is initial among all
  other cocones for the diagram $\tilde F' A_{(-)}$. Thus, the universal
  morphism of cocones between $\colim_i \tilde F' A_i$ and $\bar F (\colim_i
  A_i)$ is unique. And we can express it as the composite of cocone morphisms of
  \Cref{fig:cocone-morphism}
  \begin{figure*}[h]
    \centering
    \[
      \hss
      \colim_i \tilde F' A_i
      =
      \colim_i \tilde F \orthocatemb A_i
      \cong
      \tilde F \colim_i \orthocatemb A_i
      \xto{\tilde F (\eta_{\colim_i \orthocatemb A_i})}
      \tilde F \orthocatemb \orthorefl \colim_i \orthocatemb A_i
      \cong
      \tilde F  \orthocatemb \colim_i A_i
      =
      \tilde F' (\colim_i A_i)
      \hss
    \]
    \caption{The cocone morphism between $\colim_i \tilde F' A_i$ and $\tilde F' (\colim_i A_i)$.}
    \label{fig:cocone-morphism}
  \end{figure*}
  where the isomorphism $ \colim_i \tilde F \orthocatemb A_i \cong \tilde F
  \colim_i \orthocatemb A_i $ comes from the fact that $\tilde F$ preserves
  colimits by \Cref{prop:tildeF-pres-colimits}, the isomorphism $ \tilde F
  \orthocatemb \orthorefl \colim_i \orthocatemb A_i \cong \tilde F \orthocatemb
  \colim_i A_i $ is the image by $\tilde F \orthocatemb$ of the isomorphism
  $\orthorefl \colim_i \orthocatemb A_i \cong \colim_i A_i$ introduced at the
  beginning. And, by definition, $\tilde F'$ preserves the colimit $\colim_i A_i$
  iff the above sequence of morphisms of cocones is an isomorphism. This is the
  case iff $\tilde F (\eta_{\colim_i \orthocatemb A_i})$ is an isomorphism.
\end{proof}

\begin{proposition}
  \label{prop:Feta-isom-impl-barF-pres-colimits}
  If $\tilde F (\eta_{X})$ is an isomorphism for every $X \in \catob\cC$, the
  functor $\bar F \co \orthocat  O \to \orthocat  P$ preserves colimits.
\end{proposition}

\begin{proof}
  By \Cref{prop:barF-pres-colim-iff-tildeFpetacolim-isom}.
\end{proof}

\begin{proposition}
  \label{prop:Ftildeeta-isom}
  If $\tilde F(g)$ is an isomorphism for every $g \in O$, then for every $X \in
  \ps C$, $\tilde F(\eta_X)$ is an isomorphism.
\end{proposition}

\begin{proof}
  Let $X \in \ps C$. Since $\tilde F $ preserves colimits, we have the colimit
  \[
    \begin{tikzcd}[cramped,csbo=3.5em]
      &&&\tilde F (\orthocatemb \orthorefl X)&&&
      \\
      \tilde F  X_0
      \ar[rr,"\tilde F (\transemb_0)"{description}]
      \ar[rrru,"\tilde F (\eta_X)",dashed]
      &&
      \tilde F  X_1
      \ar[rr,"\tilde F (\transemb_1)"{description}]
      \ar[ru,dashed]
      &&
      \tilde F  X_2
      \ar[rr,"\tilde F (\transemb_2)"{description}]
      \ar[lu,dashed]
      &&
      \cdots
      \ar[lllu,dashed]
    \end{tikzcd}
  \]
  as image of the one of~\eqref{eq:reflection-def}. We now show that every
  $\tilde F (\transemb_i)$ is an isomorphism for every $i\in \N$, so that
  $\tilde F (\eta_X)$ is an isomorphism.

  So let $i \in \N$. Again, since $\tilde F $ preserves colimits, $\tilde F X_{i+1}$ is the
  colimit on the image of the diagram~\eqref{eq:step-Xi-diag} by $\tilde F $, that is,
  \begin{equation}
    \label{eq:step-HXi-diag}
    \begin{tikzcd}
      &
      \smash{\mathmakebox[0pt][r]{\raisebox{1.2em}{$\ddots$}}}\;\tilde F A^e_{i,j}\;\smash{\mathmakebox[0pt][l]{\raisebox{-1.2em}{$\ddots$}}}
      \ar[r,"\tilde F (g^e_{i,j})"]
      \ar[d,"\tilde F (f^e_{i,j})"{description}]
      &
      \tilde F B^e_{i,j}
      \\
      \smash{\mathmakebox[0pt][l]{\;\;\raisebox{1.2em}{$\vdots$}}}\smash{\mathmakebox[0pt][l]{\;\;\raisebox{-1.5em}{$\vdots$}}}\tilde F B^u_{i,j}
      \ar[r,"\tilde F (h^u_{i,j})",shift left]
      \ar[r,"\tilde F (h'^u_{i,j})"',shift right]
      &
      \tilde F X_i
    \end{tikzcd}
  \end{equation}
  with $\tilde F (\transemb_i)$ as the coprojection from $\tilde F X_i$ to
  $\tilde F X_{i+1}$. By hypothesis, $\tilde F (g^e_{i,j})$ is an isomorphism
  for every $j \in J^{i,e}$. Moreover, for every $j \in J^{i,u}$, we have
  $h^u_{i,j} \circ g^u_{i,j} = h'^u_{i,j} \circ g^u_{i,j}$. So that $\tilde F
  (h^u_{i,j}) \circ \tilde F (g^u_{i,j}) = \tilde F (h'^u_{i,j}) \circ \tilde F
  (g^u_{i,j})$, which amounts to $\tilde F (h^u_{i,j}) = \tilde F (h'^u_{i,j})$
  since $\tilde F (g^u_{i,j})$ is an isomorphism by hypothesis. Thus, a colimit
  on the diagram~\eqref{eq:step-HXi-diag} is given by $\tilde F X_i$: assuming a
  morphism $\tilde F X_i \to Y \in \orthocat P$, we are able to easily extend it
  to a unique cocone on~\eqref{eq:step-HXi-diag} since the $\tilde F
  (g^e_{i,j})$'s are isomorphisms and since the $\tilde F (h^u_{i,j})$'s are
  equal to the $\tilde F (h'^u_{i,j})$'s; conversely, every cocone with head $Y$
  includes in particular a morphism $\tilde F X_i \to Y$. Thus, $\tilde F X_i$
  and $\tilde F X_{i+1}$ are both colimits on~\eqref{eq:step-HXi-diag} so that
  $\tilde F (\transemb_i)$ is an isomorphism.
\end{proof}

\begin{theorem}
  \label{thm:criterion-la}
  Given two \lfp categories $\cC$ and $\cD$ modeled by two presheaf models $(C,
  O)$ and $(D, P)$, and a functor $\cF \co \cC \to \cD$ modeled by a Kan model
  $F \co C \to \ps D$, if the functor $\tilde F \co \ps C \to \orthocat P$ sends
  every $g \in O$ to an isomorphism in $\orthocat P$, then $\cF$ is a left
  adjoint.
\end{theorem}

\begin{proof}
  It is enough to show that $\tilde F' \co \orthocat O \to \orthocat P$
  preserves colimits. By \Cref{prop:Feta-isom-impl-barF-pres-colimits}, we just
  need that $\tilde F(\eta_X)$ be an isomorphism for every $X \in \ps C$. The
  latter is then given by \Cref{prop:Ftildeeta-isom}.
\end{proof}
Based on the criterion, we are already able to handle situations where $P =
\emptyset$, since then $\orthocat P = \ps D$ and $\tilde F \cong F'$.
\begin{example}
  \label{ex:ob-functor-model-la}
  Considering the functor $\Ob \co \Cat \to \Set$ whose Kan model was given in
  \Cref{ex:ob-functor-model}, we can apply the criterion of
  \Cref{thm:criterion-la} to check whether this functor is a left adjoint. Since
  $\psmoset = \emptyset$, the reflection functor $\orthorefl \co
  \pswide{\psmcatset} \to \orthocat{(\psmoset)}$ is the identity, and verifying
  the criterion is even simpler: we just need to compute whether $\FOb' = \Lan
  \FOb$ sends every $g \in \psmocat$ to an isomorphism. We can compute that it
  is the case (see \Cref{sec:implementation}): $g^\cstp$ is sent to a bijection
  between sets of cardinality $3$, $g^\cstlu$ and $g^\cstru$ to bijections
  between sets of cardinality $2$, $g^\cstass$ to a bijection between sets of
  cardinality $4$.
\end{example}

\begin{example}[Non-example]
  \label{ex:prod-functor-not-criterion}
  We know that the product functor $(X,Y) \mapsto X\times Y$ of sets is not a
  left adjoint, since it does not preserve coproducts for example. The criterion
  of \Cref{ex:ob-functor-model} must thus not be satisfied for the Kan model
  $\Fprod$ of \Cref{ex:fm-product}, so that the only morphism $g^{\cstp}$ of
  $\psmosetset$ must not be sent to an isomorphism by $\tilde \Fprod$. We can
  compute again (see \Cref{sec:implementation}) that it is the case: $g^{\cstp}$
  is sent to a function between the empty set and a singleton set.
\end{example}

\subsection{Computing reflections, exhaustively}

Let $\cC$ be a category and $O$ be a finite set of morphisms of $\cC$. By what
we have exposed, we are interested in knowing when a morphism $m \co X \to Y \in
\cC$ is sent to an isomorphism by the reflection functor $\orthorefl \co \cC \to
\orthocat{(\cC,O)}$. We have different ways, or strategies, through which this
can be done, and that will determine the kind of computer aid we can get.

The first one might be called the \emph{exhaustive strategy}, which actually
computes $\orthorefl(m)$ completely: we first compute the reflection $\orthorefl
Y$, together with the morphism $\eta_Y \co Y \to \orthocatemb \orthorefl Y$. We
then compute $m' = \eta_Y \circ Y \co X \to \orthocatemb \orthorefl Y$ and
finally compute the liftings of $m'$ against the chain
\[
  \begin{tikzcd}[cramped,rsbo=2.5em,csbo=3em]
    &&&
    \orthocatemb \orthorefl Y
    &&&
    \\
    X_0
    \ar[rr]
    \ar[urrr,"m'"{description}]
    &&
    X_1
    \ar[rr]
    \ar[ur,dashed]
    &&
    X_2
    \ar[rr]
    \ar[ul,dashed]
    &&
    \cdots
    \ar[ulll,dashed]
  \end{tikzcd}
\]
so that we get $\orthocatemb\orthorefl(m) \co \orthocatemb\orthorefl X \to
\orthocatemb\orthorefl Y$ since $\orthocatemb\orthorefl X$ is the colimit of
the $X_i$'s. While this method is quite straight-forward, it has some problems.
First, it is infinite as it is. Indeed, to compute both $\orthocatemb\orthorefl
X$ and $\orthocatemb\orthorefl Y$, we need to do an infinite number of steps,
which are impossible to simulate on a computer with a finite time of execution.
But, this might not always be a problem. Indeed, when $\transemb_{i_e} \co
X_{i_e} \to X_{i_e+1}$ is an isomorphism for some $i_e \in \N$, then
$\transemb_{i} \co X_{i} \to X_{i+1}$ is an isomorphism for every $i \ge i_e$,
so that $\orthocatemb\orthorefl X \cong X_{i_e}$ and we can stop the computation
of $\orthocatemb\orthorefl X$ at the step $i_e$. Same for the computation of
$\orthocatemb\orthorefl Y$. However, even when the computation of the
reflections of $X$ and $Y$ can be done in a finite time, the exhaustive method
is very costly: at each step $i$, we must compute the sets $J^{i,e}$ and
$J^{i,u}$, which is computationally expensive.

\subsection{The game of reflection}

Better strategies can be found, which basically give up on completely computing
the reflection of morphisms. Instead, they try to win the following game, in as
few moves as possible. This \eq{game} consists of a class of configurations,
among which some of them are winning, and some moves that can be played to go
from a configuration to another.
\begin{definition}[Game of reflection] Let $\cC$ be a category with all small
  colimits and $O$ be a set of morphisms of $\cC$. A \emph{configuration} of the
  game is just a morphism $m \co X \to Y$ of $\cC$. Such a configuration $m$ is
  \emph{winning} when $m$ is an isomorphism. Starting from $m$, one can make a
  \emph{move} towards another configuration using the following authorised
  moves:
  \begin{itemize}
  \item \emph{domain-existential (Dom-E) move}: given $g \co A \to B \in O$ and $f \co A
    \to X$, and $h \co B \to Y$ such that $h \circ g = m \circ f$, we move to the
    configuration $m' \co X' \to Y$, obtained as the factorisation of $m$ and $h$
    through the pushout
    \begin{equation}
      \label{eq:dom-moves-pushout}
      \begin{tikzcd}[cramped,rsbo=2.5em]
        A
        \ar[r,"g"]
        \ar[d,"f"']
        &
        B
        \ar[d,dashed]
        \\
        X
        \ar[r,dashed,"p"']
        &
        X'
      \end{tikzcd}
    \end{equation}
  \item \emph{domain-unicity (Dom-U) move}: given $g \co A \to B \in O$ and $h,h' \co B
    \to X$ such that $h \circ g = h' \circ g$ and $m \circ h = m \circ h'$, we move to the
    configuration $m' \co X' \to Y$, obtained as the factorisation of $m$
    through the coequaliser
    \[
      \begin{tikzcd}
        B
        \ar[r,shift left,"h"]
        \ar[r,shift right,"h'"']
        &
        X
        \ar[r,dashed,"q"]
        &
        X'
      \end{tikzcd}
    \]
  \item \emph{codomain-existential (Cod-E) move}: given $g \co A \to B \in O$ and
    $f \co A \to Y$, we move to the configuration $m' \co X \to Y'$ with $m' = p
    \circ m$ where $p$ is the coprojection of the pushout
    \[
      \begin{tikzcd}[cramped,rsbo=2.5em]
        A
        \ar[r,"g"]
        \ar[d,"f"']
        &
        B
        \ar[d,dashed]
        \\
        Y
        \ar[r,dashed,"p"']
        &
        Y'
      \end{tikzcd}
    \]
  \item \emph{codomain-unicity (Cod-U) move}: given $g \co A \to B \in O$ and
    $h,h' \co B \to Y$ such that $h \circ g = h' \circ g$, we move to the
    configuration $m' \co X \to Y'$ with $m' = q \circ m$ where
    \[
      \begin{tikzcd}
        B
        \ar[r,shift left,"h"]
        \ar[r,shift right,"h'"']
        &
        Y
        \ar[r,dashed,"q"]
        &
        Y'\zbox.
      \end{tikzcd}
    \]
  \end{itemize}
\end{definition}

\begin{lemma}
  \label{lem:ortho-morphs-refl-are-isos}
  Given $g \co A \to B\in O$, $\orthorefl(g) \co \orthorefl A \to \orthorefl B$
  is an isomorphism.
\end{lemma}
\begin{proof}
  Let $Z \in \orthocat {(\cC,O)}$. We have the isomorphisms
  \[
    \orthocat O(\orthorefl B,Z)
    \cong
    \cC(B,\orthocatemb Z)
    \stackrel{- \circ g}{\cong}
    \cC(A,\orthocatemb Z)
    \cong
    \orthocat O(\orthorefl A,Z)
  \]
  so that, by Yoneda lemma, $\orthorefl A \cong \orthorefl B$. Moreover, the
  isomorphism is given as the image of $\unit{\orthorefl B} \in \orthocat
  O(\orthorefl B,\orthorefl B)$ in $\orthocat O(\orthorefl A,\orthorefl B)$
  through the above chain of isomorphisms. It can be checked that it is
  $\orthorefl(g)$.
\end{proof}

\begin{proposition}
  Let $m \co X \to Y \in \cC$ be a morphism seen as a configuration of the above
  game, and let $m'$ be a morphism be obtained after playing one of the
  authorised moves. Then, $\orthorefl(m)$ is an isomorphism if and only if
  $\orthorefl(m')$ is an isomorphism.
\end{proposition}
\begin{proof}
  It is enough to do this for Dom-E moves, since the other cases are very
  similar. Thus, let's consider the play of a Dom-E move from a morphism $m \co
  X \to Y$:
  \[
    \begin{tikzcd}[cramped,rsbo=2.5em]
      A
      \ar[r,"g"]
      \ar[d,"f"']
      &
      B
      \ar[d,dashed]
      \\
      X
      \ar[r,dashed,"p"{description}]
      \ar[d,"m"']
      &
      X'
      \ar[ld,"m'"]
      \\
      Y
    \end{tikzcd}
  \]
  To conclude, we just need to show that $\orthorefl(p)$ is an isomorphism. But,
  since $\orthorefl$ preserves pushouts, it is a consequence of the fact that
  $\orthorefl(g)$ is an isomorphism by \Cref{lem:ortho-morphs-refl-are-isos}.
\end{proof}

Thus, to show that some $m \co X \to Y$ is sent to an isomorphism through
$\orthorefl$, it is enough to play the game of reflection, starting from $m$ and
playing zero, one or several moves, until we hit a winning configuration, that
is, an isomorphism in $\cC$. This can take the form of an automated method as a
\emph{heuristic strategy}, where the computer tries to play this game itself,
trying to apply the moves efficiently to hit a winning configuration as soon as
possible. Or it can take the form of a manual method, where the user plays
themselves the game. A computer can still be useful in this last method to check
that the sequence of moves is correct.

\begin{example}
  \label{ex:cat-prod-with-2}
  Consider the Kan model of \Cref{ex:times-functor-model} in the case where
  $\cC = \Cat$, $(C,O) = (\psmcatcat,\psmocat)$ and $B = \mathbf{2}$, the
  category with two objects and only one non-trivial \eq{walking arrow} between
  them, and $Y$ is an object of $\pswide \psmcatcat$ corresponding to
  $\mathbf{2}$. In order to prove that $-\times \mathbf{2}$ is a left adjoint,
  we just need to show that $\tilde \Ftimes$ sends each of $g^\cstp$,
  $g^\cstlu$, $g^\cstru$, $g^\cstass$ to isomorphisms. Concerning $g^\cstlu$,
  since $\Ftimes'(g^\cstlu)$ is essentially $g^\cstlu \times Y \co
  A^\cstlu\times Y \to B^\cstlu \times Y \in \pswide \psmcatcat$, we can observe
  (and compute!) that $\Ftimes'(g^\cstlu)$ fails to be an isomorphism because
  three elements of $B^\cstlu(\cstm)$ have two antecedents each: this is simply
  three copies (because $Y(\cstm)$ has three elements) of the situation observed
  on $g^\cstlu$ alone, where $u$ and $w$ in $A^\cstlu$ are both sent to $u' \in
  B^\cstlu$. In order, to make $\Ftimes'(g^\cstlu)$ into an isomorphism, we just
  need to apply three Dom-E moves with $g^\cstlu$ itself to unify each copy of
  the pair $u,w$ in $A^\cstlu\times Y$. The cases of $g^\cstru$, $g^\cstass$ and
  $g^\cstp$ are similar.
\end{example}

\subsection{A criterion for cartesian closure}

From our criterion for left-adjointness, a criterion for cartesian closure can
be derived. Recall that a category with finite products $\cC$ is \emph{cartesian
  closed} when for all objects $A,B \in \cC$, there is an object $B^A$ and a
morphism $\ev_{A,B} \co A \times B^A \to B$ satisfying the following universal
property: given an object $X \in \cC$ and a morphism $f \co A \times X \to B$,
there is a unique morphism $\bar f \co X \to B^A$ such that $f = \ev_{A,B} \circ
(A \times \bar f)$. A standard characterisation of cartesian closure, which is
essentially a reformulation of the above definition, is
\begin{proposition}
  \label{prop:charact-closure}
  A category with finite products $\cC$ is cartesian closed iff for every $A \in
  \cC$, the functor $A \times (-) \co \cC \to \cC$ is a left adjoint.
\end{proposition}
In the case of \lfp categories modeled by presheaf models, we can already use
the criterion of \Cref{thm:criterion-la} to prove that the functors $A \times
(-) \co \cC \to \cC$ are left adjoints. But this gives us an infinite number of
instances to check, which cannot be done by a computer. However, a presheaf
model $(C,O)$ for $\cC$ allows us to reduce the number of instances to check to
one instance per object of $C$:
\begin{theorem}
  Given an \lfp category $\cC$ modeled by a presheaf model $(C,O)$, $\cC$ is
  cartesian closed if for every $c \in \Ob(C)$, the functor $\orthorefl(\yoneda_{C}(c)
  \times \orthocatemb (-)) \co \orthocat {(\ps C,O)} \to \orthocat {(\ps C,O)}$
  is a left adjoint.
\end{theorem}
\begin{proof}
  By \Cref{prop:charact-closure}, we just need to show that, given $A \in
  \orthocat O \simeq \cC$, the functor $A \times (-) \co \orthocat O \to
  \orthocat O$ preserves colimits. By \Cref{prop:density-for-product}, and the
  fact that the functors $(-) \otimes (\orthocatemb A)(c)$ and the coend
  construction adequately preserves colimits, this is a consequence of the
  hypothesis.
\end{proof}
\noindent Then our criterion for left adjointness applies, since:
\begin{proposition}
  Given a presheaf model $(C,O)$ and $c \in C$, the functor
  $\orthorefl(\yoneda_{C}(c) \times \orthocatemb (-)) \co \orthocat {(\ps C,O)}
  \to \orthocat {(\ps C,O)}$ can be modeled by a Kan model.
\end{proposition}
\begin{proof}
  Given $B \in \orthocat O$, by the density formula
  (\Cref{prop:density-formula}), we have
  \[
    \begin{aligned}
      \orthorefl(\yoneda_{C}(c) \times \orthocatemb B)
      &
      \cong
      \orthorefl(\yoneda_{C}(c) \times \int^{d \in C} \yoneda_C(d) \otimes \orthocatemb B(d))
        \\
      &\cong
      \orthorefl(\int^{d \in C} (\yoneda_{C}(c) \times \yoneda_C(d)) \otimes \orthocatemb B(d))
    \end{aligned}
  \]
  since the functor $\yoneda_{C}(c) \times (-)$ commutes with colimits by the
  cartesian closure of $\ps C$. Thus, an adequate Kan model for the functor of the
  statement is given by $F \co C \to \ps C$ defined by $F(-) = \yoneda_C(c)
  \times \yoneda_C(-)$.
\end{proof}

\section{An implementation}

\label{sec:implementation}

We now discuss how one can implement the methods we introduced, which rely on
computations on presheaves and presheaf morphisms. One of the main concern for
such an implementation, in addition to the actual used algorithms, is the one of
\emph{encoding}, that is, of the representation the instances of structures we
consider as finite elements of some datatypes. Under adequate finiteness
conditions, both encodings and computations for our structures can be derived
from the ones on finite sets and finite functions, which can be easily
represented.

Since our prototype implementation is written in \ocaml, our explanations will
be illustrated by syntax and types close to the ones of this language.

\subsection{Encoding finite sets}

A crucial point to note about our methods we introduced is that all
constructions we need (left Kan extension, pushouts, coequalisers, \etc) are
defined up to isomorphism. Similarly, the arguments of these constructions can
be given up to isomorphism.

Thus, in order to build an implementation of our methods on computations on
finite sets, it is enough to represent sets up to isomorphism. But a finite set
$S$ is, up to isomorphism, a finite subset of $S' \subset \N$ or, more
generally, any finite subset of inhabitants of some infinite datatype. In this
regard, we define the type \texttt{sgen = SGen of int} to encode the elements of
the sets we consider (we gloss over the fact that \texttt{int} is technically a
datatype with finite number of inhabitants since it is \eq{big enough} in
practice; the concerned reader can replace \texttt{int} by a type for \emph{big
  integers}). The \ocaml standard library then provides a module \texttt{Set} to
represent sets of element of some type that need to be specified in advance. We
instantiate this module for the type \texttt{sgen} to obtain a module
\texttt{SSet}. The finite sets will then be encoded by inhabitants of the type
\texttt{set = SSet.t}.

Then, one needs to represent the functions $f$ between two sets $A$ and $B$
encoded by two terms \ocamltoinline{a} and \ocamltoinline{b} of type
\ocamltoinline{set}. One natural idea is to represent them as \ocaml functions
\ocamltoinline{f : sgen -> sgen} which would only be defined for the elements of
type \ocamltoinline{sgen} that belongs to \ocamltoinline{a}, and would return an
element of type \ocamltoinline{sgen} that belongs to \ocamltoinline{b}. However,
we want to incrementally construct such functions and to be able to inspect
them, which would be inefficient using \ocaml functions. Instead, we use the
\ocamltoinline{Map} module which provides an efficient datastructure to
represent a set of (key,value) bindings, that will use to represent the graph
$\set{(a,f(a)) \mid a \in A}$ of the functions $f$ we consider. We instantiate
this module for the keys of type \texttt{sgen} to obtain a module \texttt{SMap}.
The functions between two finite sets are then encoded by inhabitants of the
type \texttt{sfun = sgen SMap.t}.

\subsection{Encoding finite categories}

We can now encode finite categories, that is, categories with a finite number of
objects and morphisms. Such finite categories can be encoded by the record
\begin{verbatim}
type cat = {
  objs : set; morphs : set;
  src : sgen -> sgen; tgt : sgen -> sgen;
  id : sgen -> sgen; comp : sgen -> sgen -> sgen}
\end{verbatim}
Here, we prefer to use the \ocaml function type to represent functions, since we
will not need to inspect or change the operations of a category $C$ once it is
introduced. The operation \texttt{id} will only be defined for those
\texttt{sgen} that are members of \texttt{objs}, the operation \texttt{src} will
only defined for those \texttt{sgen} that are members of \texttt{morphs},
\etcend From these encodings, one can easily write a procedure to check whether
the unitality and associativity axioms of categories are satisfied.

\subsection{Encoding presheaves}

Now that we have encodings for finite sets and finite categories, one can
define an encoding for the finite presheaves over a finite category $C$. It will
be a record
\begin{verbatim}
type ps = { ps_sets : set SMap.t ;
            ps_maps : sfun SMap.t }
\end{verbatim}
where \texttt{ps\_sets} encodes the map $c \in \Ob(C) \mapsto X(c)$ of a finite
presheaf $X$ (the keys of the \texttt{SMap.t} are encoding the objects of $C$),
and \texttt{ps\_maps} encodes the map $(f \co c \to c' \in C) \mapsto (X(f) \co
X(c') \to X(c))$ (the keys of the \texttt{SMap.t} are encoding the morphisms of
$C$).

We can now encode morphisms between two presheaves $X$ and $Y$ using the type
\begin{verbatim}
type ps_morph = sfun SMap.t
\end{verbatim}
where the keys of the \texttt{SMap.t} are encoding the objects of $C$, and the
value of this \texttt{SMap.t} are the encoding of the functions $X(c) \to Y(c)$
for $c \in \Ob(C)$. We can easily write some procedure to check that an
inhabitant of \texttt{ps\_morph} is indeed a correct encoding of a morphism
between $X$ and $Y$.

Using this encoding, we will be able to encode the \emph{finite} presheaf
models:
\begin{definition}
  A presheaf model $(C,O)$ is \emph{finite} when $C$ is a finite category and
  $O$ is a finite set of morphisms $f \co A \to B$ where both $A$ and $B$ are
  finite.
\end{definition}
Indeed, we can use for example the record:
\begin{verbatim}
type ps_model = {
  psm_cat : cat ;
  psm_omorphs : (ps * ps * ps_morph) list }
\end{verbatim}
where \texttt{psm\_cat} encodes a category $C$ and \texttt{psm\_omorphs} encodes a
list of triples $(A \in \ps C, B \in \ps C, f \co A \to B \in \ps C)$, that
encodes the set $O$ of orthogonality conditions of a presheaf model $(C,O)$.

Finite Kan models (for an adequate notion of \emph{finiteness}) can be encoded
using similar ideas.

\subsection{Computing the operations}
\label{sec:computing-operations-brief}

With the above encodings, one can write the code for the different operations we
need. First we can derive a procedure for computing colimits of finite
presheaves from a procedure to compute the colimits of finite sets. Indeed, one
can easily compute a finite coproduct of finite sets, and coequalisers of finite
sets, which is all that is required to compute more general finite colimits of
finite sets. Then, we can implement a procedure to compute the action of the
functor $F' \co \ps C \to \ps D$ from an encoded Kan model $F \co C \to \ps D$.
Indeed, for a given finite presheaf $X$, $F'(X)$ is just an adequate finite
colimit of finite presheaves which can be computed. Finally, we can write the
procedures for the operations needed to find and play the moves of the game of
reflection. See \Cref{sec:details-implementation} for details.

\subsection{Code for the examples}

We can now consider the code for the examples of this paper. For technical and
efficiency reasons, the actual datatypes used that we use are different from the
ones that we just discussed. One of the main difference is that we prefer to
encode finite categories using finite presentations, since it felt easier
describe examples this way, than to specify their whole set of morphisms. The
files \texttt{setCat.ml}, \texttt{categoryCat.ml}, \texttt{pairCat.ml} encodes
the presheaf models for $\Set$ as in \Cref{ex:pm-set}, $\Set \times \Set$ as in
\Cref{ex:pm-set-times-set}, and $\Cat$ as in \Cref{ex:pm-cat} respectively.

Then, the files \texttt{truncation\_functor.ml}, \texttt{pair\_functor.ml} and
\texttt{product\_w\_C1\_functor.ml} encode the Kan models of the functors of
\Cref{ex:ob-functor-model,ex:fm-product,ex:cat-prod-with-2} and compute the
criterion for left adjointness using the exhaustive search for the reflection
procedure. While the two first examples terminate, and are able to give us
information discussed in
\Cref{ex:ob-functor-model-la,ex:prod-functor-not-criterion} about the computed
reflections, the computation does not terminate in reasonable time for the last
file. Instead, we can manually play the game of reflection and show that we can
win the game in the cases of $g^\cstlu$, $g^\cstru$ and $g^\cstass$ in few
steps. The case of $g^\cstp$ cannot be handled by our current implementation,
since it does not allow us to choose which specific instance of a move to play.
Instead, we can only instruct our program to execute all the moves of a certain
type, which results in a looping behavior in the case of $g^\cstp$. This
alternative computation is carried out in
\texttt{product\_w\_C1\_functor\_unitl\_assoc\_only.ml}.

\clearpage

\printbibliography

@book{adamek1994locally,
  title={Locally Presentable and Accessible Categories},
  author={{Adámek, Jiří  and Rosický, Jiří}},
  number={189},
  year={1994},
  series={London Mathematical Society Lecture Notes Series},
  publisher={Cambridge University Press}
}

@book{kelly1982basic,
  title={Basic Concepts of Enriched Category Theory},
  author={Kelly, G. Max},
  number={64},
  year={1982},
  series={London Mathematical Society Lecture Notes Series},
  publisher={Cambridge University Press}
}

@book{mac2013categories,
  title={Categories for the Working Mathematician},
  author={MacLane, Saunders},
  number={5},
  year={2013},
  series={Graduate Texts in Mathematics},
  publisher={Springer}
}

@book{Loregian2021,
  doi       = {10.1017/9781108778657},
  url       = {https://doi.org/10.1017/9781108778657},
  year      = {2021},
  month     = jun,
  publisher = {Cambridge University Press},
  author    = {F. Loregian},
  title     = {(Co)end Calculus}
}

@inbook{Walters_1992,
  place={Cambridge},
  series={Cambridge Computer Science Texts},
  title={Computational Category Theory},
  booktitle={Categories and Computer Science},
  publisher={Cambridge University Press},
  author={Walters, R. F. C.},
  year={1992},
  pages={143–160},
  collection={Cambridge Computer Science Texts}
}

@article{meyers2022fast,
  title={Fast left kan extensions using the chase},
  author={Meyers, Joshua and Spivak, David I and Wisnesky, Ryan},
  journal={Journal of Automated Reasoning},
  volume={66},
  number={4},
  pages={805--844},
  year={2022},
  publisher={Springer}
}

@inproceedings{carmody1991computing,
  title={Computing quotients of actions of a free category},
  author={Carmody, S and Walters, RFC},
  booktitle={Category Theory: Proceedings of the International Conference held in Como, Italy, July 22--28, 1990},
  pages={63--78},
  year={1991},
  organization={Springer}
}

@article{CARMODY1995459,
  title = {The Todd-Coxeter Procedure and Left Kan Extensions},
  journal = {Journal of Symbolic Computation},
  volume = {19},
  number = {5},
  pages = {459-488},
  year = {1995},
  issn = {0747-7171},
  doi = {https://doi.org/10.1006/jsco.1995.1027},
  url = {https://www.sciencedirect.com/science/article/pii/S0747717185710279},
  author = {S. Carmody and M. Leeming and R.F.C. Walters}
}

@book{borceux1994handbook2,
  maintitle={Handbook of Categorical Algebra},
  author={Borceux, Francis},
  series={Encyclopedia of Mathematics and its Applications},
  number={51},
  volume=2,
  title={Categories and Structures},
  year={1994},
  publisher={Cambridge University Press}
}

@article{MOGGI199155,
title = "Notions of computation and monads",
journal = "Information and Computation",
volume = "93",
number = "1",
pages = "55 - 92",
year = "1991",
%note = "Selections from 1989 IEEE Symposium on Logic in Computer Science",
%issn = "0890-5401",
%doi = "https://doi.org/10.1016/0890-5401(91)90052-4",
%url = "http://www.sciencedirect.com/science/article/pii/0890540191900524",
author = "Eugenio Moggi",
abstract = "The λ-calculus is considered a useful mathematical tool in the study of programming languages, since programs can be identified with λ-terms. However, if one goes further and uses βη-conversion to prove equivalence of programs, then a gross simplification is introduced (programs are identified with total functions from values to values) that may jeopardise the applicability of theoretical results. In this paper we introduce calculi, based on a categorical semantics for computations, that provide a correct basis for proving equivalence of programs for a wide range of notions of computation."
}

@book{lambek1988introduction,
  title={Introduction to higher-order categorical logic},
  author={Lambek, Joachim and Scott, Philip J},
  volume={7},
  year={1988},
  publisher={Cambridge University Press}
}

@online{haskellws,
  author = {the \textsc{Haskell} Community},
  title = {{The \textsc{Haskell} programming language}},
  @year = 1999,
  url = {https://www.haskell.org/},
  @urldate = {2010-09-30}
}

@book{seely1987linear,
  title={Linear logic,*-autonomous categories and cofree coalgebras},
  author={Seely, Robert AG},
  year={1987},
  publisher={Ste. Anne de Bellevue, Quebec: CEGEP John Abbott College}
}

@article{vicary2018globular,
  title={Globular: an online proof assistant for higher-dimensional rewriting},
  author={Vicary, Jamie and Kissinger, Aleks and Bar, Krzysztof},
  journal={Logical Methods in Computer Science},
  volume={14},
  year={2018},
  publisher={Episciences.org}
}

@misc{heidemann10homotopy,
  title={homotopy.io, 2019},
  author={Heidemann, Lukas and Hu, Nick and Vicary, Jamie},
  url={https://doi. org/10.5281/zenodo},
}

@inproceedings{finstermimram17,
  author={Finster, Eric and Mimram, Samuel},
  booktitle={2017 32nd Annual ACM/IEEE Symposium on Logic in Computer Science (LICS)}, 
  title={A type-theoretical definition of weak $\omega$-categories}, 
  year={2017},
  volume={},
  number={},
  pages={1-12},
  doi={10.1109/LICS.2017.8005124}
}

@book{borceux1994handbook,
  title={Handbook of categorical algebra: volume 1, Basic category theory},
  author={Borceux, Francis},
  volume={1},
  year={1994},
  publisher={Cambridge University Press}
}

@article{schultz2017algebraic,
  title={Algebraic data integration},
  author={Schultz, Patrick and Wisnesky, Ryan},
  journal={Journal of Functional Programming},
  volume={27},
  pages={e24},
  year={2017},
  publisher={Cambridge University Press}
}

@article{schultz2016algebraic,
  title={Algebraic databases},
  author={Schultz, Patrick and Spivak, David I and Vasilakopoulou, Christina and Wisnesky, Ryan},
  journal={arXiv preprint arXiv:1602.03501},
  year={2016}
}

@inproceedings{spivak2015relational,
  title={Relational foundations for functorial data migration},
  author={Spivak, David I and Wisnesky, Ryan},
  booktitle={Proceedings of the 15th Symposium on Database Programming Languages},
  pages={21--28},
  year={2015}
}

@article{bush2003computing,
  title={Computing left Kan extensions},
  author={Bush, Michael R and Leeming, M and Walters, Robert FC},
  journal={Journal of symbolic computation},
  volume={35},
  number={2},
  pages={107--126},
  year={2003},
  publisher={Elsevier}
}

@article{posur2021constructive,
  title={A constructive approach to Freyd categories},
  author={Posur, Sebastian},
  journal={Applied Categorical Structures},
  volume={29},
  number={1},
  pages={171--211},
  year={2021},
  publisher={Springer}
}

@article{Posur2019MethodsOC,
  title={Methods of constructive category theory},
  author={Sebastian Posur},
  journal={Representations of Algebras, Geometry and Physics},
  year={2019}
}

@article{lynch2024GAT,
  author = {Lynch, Owen and Brown, Kris and Fairbanks, James and Patterson, Evan},
  year = {2024},
  month = {12},
  pages = {},
  title = {GATlab: Modeling and Programming with Generalized Algebraic Theories},
  volume = {4},
  journal = {Electronic Notes in Theoretical Informatics and Computer Science},
  doi = {10.46298/entics.14666}
}

@article{freyd1972aspects,
  title={Aspects of topoi},
  author={Freyd, Peter},
  journal={Bulletin of the Australian Mathematical Society},
  volume={7},
  number={1},
  pages={1--76},
  year={1972},
  publisher={Cambridge University Press}
}

@article{cartmell1986generalised,
  title={Generalised algebraic theories and contextual categories},
  author={Cartmell, John},
  journal={Annals of pure and applied logic},
  volume={32},
  pages={209--243},
  year={1986},
  publisher={Elsevier}
}

@article{ToddCoxeter1936,
  title={A practical method for enumerating cosets of a finite abstract group},
  volume={5},
  DOI={10.1017/S0013091500008221},
  number={1},
  journal={Proceedings of the Edinburgh Mathematical Society},
  author={Todd, J. A. and Coxeter, H. S. M.},
  year={1936},
  pages={26–34}
}

\clearpage

\appendix

\subsection{Essentially algebraic theories}

\label{sec:eat}

Here, we provide some details about the notion of \emph{essentially algebraic
  theories}, which provides a characterisation of \lfp categories. It is closely
related to the notion of generalised algebraic theories, which are basically an
extension of the notion of essentially algebraic theories which allows for
dependent typing in the definition of operations. We refer the reader to
standard references like \cite[Section 3.D]{adamek1994locally} for a longer
presentation.
\begin{definition}
  Given a set~$S$, an \emph{$S$\sorted signature} is the data of a set~$\Sigma$
  of \emph{symbols}
  \[
    \sigma \co s_1 \times \cdots \times s_n \to s
  \]
  where $s_1 \times \cdots \times s_n$ is the \emph{input arity}, with each
  $s_i$ in $S$, and where $s \in S$ is the \emph{output sort}. From such a
  signature and a \emph{context} $\Gamma = (x_1 : s_1, \ldots, x_k : s_k)$ of
  distinct variables typed by sorts $s_i \in S$, one can build \emph{terms} the
  expected way, as trees whose inner nodes are made of symbols of $\Sigma$, and
  whose leaves are variables $x_i : s_i \in \Gamma$, all respecting arities and
  sorts. The \emph{output sort} of such a term is the output sort of the root
  node, or, if there are no nodes, the sort of the only variable. We write
  $\freeterms\Sigma(\Gamma)$ for this set of terms.
\end{definition}

\begin{definition}
  An \emph{essentially algebraic theory} is a tuple
  \[
    \stdtheory = (S,\Sigma,E,\Sigma_t,\Def)
  \]
  where
  $S$ is a set,
  $\Sigma$ is an $S$\sorted signature,
  $E$ is a set of triples $(\Gamma,t_1,t_2)$ where~$t_1,t_2$ are terms of
    $\freeterms\Sigma(\Gamma)$ with the same output sort,
  $\Sigma_t$ is a subset of~$\Sigma$,
  $\Def$ is a function which maps~$\sigma \co s_1 \times \cdots \times s_n
    \to s\in \Sigma\setminus\Sigma_t$ to a set of equations~$t_1 = t_2$
    where~$t_1,t_2$ are terms of $\freeterms\Sigma_t(x_1 : s_1, \ldots, x_k :
    s_k)$.
\end{definition}
In the above definition, the set~$S$ represents the different \emph{sorts} of
the theory, the set~$\Sigma$ the different operations that appear in the theory,
the set~$E$ the global equations satisfied by the theory, the set~$\Sigma_t$ the
operations whose domains are total, and the function~$\Def$ the equations that
define the domains of the partial operations.

\begin{definition}
  Given such an essentially algebraic theory~$\stdtheory$, a \emph{premodel
    of~$\stdtheory$}, or~\emph{$\stdtheory$\premodel}, is the data of
  \begin{itemize}
  \item for all~$s \in S$, a set~$M_s$,
    
  \item for all~$\sigma \co s_1 \times \cdots \times s_n \to s \in \Sigma_t$, a function
    \[
      M_\sigma \co M_{s_1} \times \cdots \times M_{s_n} \to M_s\zbox,
    \]
    
  \item for all~$\sigma \co s_1 \times \cdots \times s_n \to s \in
    \Sigma\setminus\Sigma_t$, a partial function
    \[
      M_\sigma \co M_{s_1} \times \cdots \times M_{s_n} \to M_s\zbox.
    \]
  \end{itemize}
  Given such a premodel, a context $\Gamma = (x_1:s_1,\ldots,x_k:s_k)$, a term
  $t \in \freeterms\Sigma(\Gamma)$ of output sort $s$, and $\bar y =
  (y_1,\ldots,y_n) \in M_{s_1} \times \cdots \times M_{s_n}$, the
  \emph{interpretation} $\eateval{t}_{\bar y}$ is either not defined or an
  element of $M_s$ according to the following rules:
  \begin{itemize}
  \item if~$t = x_i$ for some~$i \in \set{1,\ldots,k}$, then $\eateval t_{\bar
      y}$ is defined and
    \[
      \eateval t_{\bar y} = y_i\zbox,
    \]
    
  \item if~$t = \sigma (t_1,\ldots,t_n)$ where $n$ is the arity of $\sigma$, and
    terms $t_1,\ldots,t_n \in \freeterms\Sigma(\Gamma)$, then~$\eateval t_{\bar y}$ is
    defined if and only if $\eateval {t_1}_{\bar y},\ldots,\eateval {t_n}_{\bar
      y}$ are defined and $M_\sigma$ is defined at~$\eateval {t_1}_{\bar
      y},\ldots,\eateval {t_n}_{\bar y}$ and, in this case,
    \[
      \eateval t_{\bar y} =
      M_\sigma(\eateval {t_1}_{\bar y},\ldots,\eateval {t_n}_{\bar y})\zbox.
    \]
  \end{itemize}
  A premodel of~$\stdtheory$ is then a \emph{model of $\stdtheory$}, or
  \emph{$\stdtheory$\model}, when moreover
  \begin{itemize}
  \item for all~$\sigma \co s_1 \times \cdots \times s_n \to s \in
    \Sigma\setminus\Sigma_t$, $M_\sigma$ is defined at $\bar y =
    (y_1,\ldots,y_n) $ in $M_{s_1} \times \cdots \times M_{s_n}$ if and only if,
    for all~$(t_1,t_2) \in \Def(\sigma)$, we have~$\eateval{t_1}_{\bar y} =
    \eateval{t_2}_{\bar y}$,
  \item for every triple~$(\Gamma,t_1,t_2) \in E$ where~$\Gamma =
    (x_1:s_1,\ldots,x_n:s_n)$, given a tuple~$\bar y = (y_1,\ldots,y_n) \in
    M_{s_1} \times \cdots \times M_{s_n}$, if both~$\eateval {t_1}_{\bar y}$
    and~$\eateval {t_2}_{\bar y}$ are defined, then~$\eateval {t_1}_{\bar y} =
    \eateval {t_2}_{\bar y}$.
  \end{itemize}
  A \emph{morphism} between two models $M$ and $M'$ is a family of functions
  \[
    F = (F_s \co M_s \to M'_s)_{s \in S}
  \]
  such that, for every $\sigma \co s_1 \times \cdots \times s_n \to s \in
  \Sigma$, for every $\bar y = (y_1,\ldots,y_n)$ in $M_{s_1} \times \cdots
  \times M_{s_n}$ such that $M_\sigma(y_1,\ldots,y_n)$ is defined, we have that
  the element $M'_\sigma(F_{s_1}(y_1),\ldots,F_{s_n}(y_n))$ is defined and
  \[
    F_s(M_\sigma(y_1,\ldots,y_n)) =
    M'_\sigma(F_{s_1}(y_1),\ldots,F_{s_n}(y_n))\zbox.
  \]
  We write $\Mod(\stdtheory)$ for the category of $\stdtheory$\models and
  morphisms.
\end{definition}
Essentially algebraic theories allows one to easily recognise a lot of
categories as \lfp using the following property:
\begin{theoremapp}[{\cite[Theorem 3.36]{adamek1994locally}}]
  \label{thm:lfp-cats-as-eat-models}
  A category $\cC$ is \lfp if and only if it is equivalent to the category
  $\Mod(\stdtheory)$ for some essentially algebraic theory $\stdtheory$.
\end{theoremapp}

\begin{example}
  The category~$\Set$ is essentially algebraic since it is the category of
  models of the essentially algebraic
  theory~$(\set \csts,\emptyset,\emptyset,\emptyset,\emptyset)$.
\end{example}
\begin{example}
  \label{ex:mon-ess-alg-theory}
  The category~$\Mon$ of monoids and monoid morphisms is the category of models
  of the essentially algebraic theory
  \[
    \stdtheory^{\mathrm{mon}} = (\set{\csts},\set{\cste \co \termobj \to {\csts},\cstm \co {{\csts}} \times {{\csts}} \to {\csts}},E,\set{\cste,\cstm},\emptyset)
  \]
  where~$E$ consists of three equations
  \begin{itemize}
  \item $\cstm (\cste,x_1) = x_1$ in the context~$(x_1 \co {\csts})$,
  \item $\cstm (x_1,\cste) = x_1$ in the context~$(x_1 \co {\csts})$,
  \item $\cstm (\cstm(x_1,x_2),x_3) = \cstm(x_1,\cstm(x_2,x_3))$ in the context~$(x_1 \co
    {{\csts}},x_2 \co {{\csts}},x_3 \co {{\csts}})$.
  \end{itemize}
  In particular,~$\Mon$ is locally finitely presentable.
\end{example}
\begin{example}
  \label{ex:cat-ess-alg-theory}
  The category~$\Cat$ of small categories is the category of models of the
  essentially algebraic theory~$\stdtheory^{\mathrm{cat}} =
  (S,\Sigma,E,\Sigma_t,\Def)$ defined as follows. The set~$S$ consists of two
  sorts~$\csto$ and~$\cstm$ corresponding to objects and morphisms, and
  \[
    \Sigma = \set{\cstsrc \co \cstm \to \csto,\; \csttgt \co \cstm \to
      \csto,\;\cstid \co \csto \to \cstm,\; \cstcomp\co \cstm \times \cstm \to
      \cstm}\zbox.
  \]
  Moreover,~$E$ consists of the equations
  \begin{itemize}
  \item $\cstsrc (\cstid (x_1)) = x_1$ and~$\csttgt(\cstid (x_1)) =
    x_1$ in the context~$(x_1 \co \csto)$,
    
  \item $\cstsrc(\cstcomp(x_1,x_2)) = \cstsrc(x_1)$ and~$\csttgt(\cstcomp(x_1,x_2)) =
    \csttgt(x_2)$ in the context~$(x_1 \co \cstm,x_2 \co \cstm)$,
    
  \item $\cstcomp(\cstid(\cstsrc(x_1)),x_1) = x_1$
    and~$\cstcomp(x_1,\cstid(\csttgt(x_1))) = x_1$ in the context~$(x_1 \co
    \cstm)$,
    
  \item $\cstcomp(\cstcomp(x_1,x_2),x_3) = \cstcomp(x_1,\cstcomp(x_2,x_3))$ in the
    context~$(x_1 \co \cstm,x_2 \co \cstm,x_3 \co \cstm)$.
  \end{itemize}
  Finally,~$\Sigma_t = \set{\cstsrc,\csttgt,\cstid}$, and~$\Def(\cstcomp)$ is
  the singleton set containing the equation~$\csttgt(x_1) = \cstsrc(x_2)$. This
  shows that~$\Cat$ is a locally finitely presentable category.
\end{example}

\subsection{Details about the implementation of the operations}
\label{sec:details-implementation}

Here, we give additional details about the implementation of the operations that
were briefly discussed in \Cref{sec:computing-operations-brief}, like the
computation of finite colimits of finite presheaves, of left Kan extensions and
of the operations of the game of reflection; the code for the latter operations
can also be used to compute the reflection using the exhaustive strategy.

\subsubsection{Computing colimits}

One of the main operations that we will need to compute in our methods is the
one of colimit of presheaves. Indeed, the Kan extension we need to build the
functor $F'$ from a Kan model $F \co C \to \ps D$ can be expressed as a
particular colimit, and the moves made in the game of reflection are computed
using pushouts and coequalisers.

We will only consider the more elementary computation of colimits of finite
sets, since a method for computing colimits of sets readily extend to
presheaves, because colimits of presheaves are pointwisely computed.

First, remember that any colimit in a cocomplete category can be computed as a
coequaliser:
\begin{proposition}
  \label{prop:colimit-as-coequalizer}
  Given a category $\cC$ with coproducts, a small category $I$ and a diagram $d
  \co I \to \cC$, the colimit $\colim_i d(i) \in \cC$, when it exists, can be
  computed as the coequaliser
  \[
    \begin{tikzcd}[csbo=6em]
      \coprod_{f \co i \to i' \in I} d(i)
      \ar[r,shift left,"{[\iota_{i} \circ \unit{d(i)}]_{f \co i \to i' \in I}}"]
      \ar[r,shift right,"{[\iota_{i} \circ f]_{f \co i \to i' \in I}}"']
      &[6em]
      \coprod_{i \in I} d(i)
      \ar[r,"q",dashed]
      &
      L
    \end{tikzcd}
  \]
  where the $\iota_i \co d(i) \to\coprod_{i \in I} d(i)$'s are the coprojections
  and $[-]_{f \co i \to i'\in I}$ is the copairing operation for the coproduct
  $\coprod_{f \co i \to i' \in I} d(i)$. The coprojections $p_i \co d(i) \to
  \colim_i d(i)$ are then, up to the isomorphism $\colim_i d(i) \cong L$, the
  morphisms $q \circ \iota_i$.
\end{proposition}
Computing the coproduct $S$ of a finite list $S_1,\ldots,S_n$ of sets is easy:
given the list $\mathtt{[s_1;\ldots;s_n]}$ of \texttt{set}s that encode these
sets, we just construct a term \texttt{s} of type \texttt{set} that encode a set
$S$, starting from an empty set \texttt{set}, and adding fresh copies of the
elements of $\mathtt{s_1},\ldots,\mathtt{s_n}$. In the process, we also
construct a sequence of maps $\mathtt{[m_1;\ldots;m_n]}$ that encode the
coprojections $S_i \to S$, by remembering, for each element of $\mathtt{s_i}$,
which fresh variable we took in $\mathtt{s}$ as a copy.

Now, computing a coequaliser
$\begin{tikzcd}
  S
  \ar[r,"f",shift left]
  \ar[r,"g"',shift right]
  &
  T
  \ar[r,dashed,"h"]
  &
  L
\end{tikzcd}$ of finite sets $S$ and $T$ is also easy. Indeed, remember that,
in this situation, $L$ can be defined as the quotient $T_{/\sim}$ of $T$ by
the smallest equivalence relation $\sim$ satisfying $f(s) \sim g(s)$ for every
$s \in S$. In order to compute $T_{/\sim}$, it is enough to compute the
equivalence classes of $T$ for $\sim$. But, given encodings $\mathtt{s}$,
$\mathtt{t}$, $\mathtt{f}$ and $\mathtt{g}$ for $S$, $T$, $f$ and $g$
respectively, this can be easily done using a standard Union–Find algorithm.
An encoding of the projection $T \to T_{/\sim}$ can then be easily computed.

Thus, we are able to compute all colimit of finite sets over a finite diagram
$I$. Since colimits on presheaves are pointwisely computed, this procedure can
be extended to compute the colimits over finite diagrams $I$ of finite
presheaves $X \in \ps C$ for a finite category $C$.

\subsubsection{Computing left Kan extensions}

Given finite categories $C,D$, and a functor $F \co C \to \ps D$ such that
$F(c)$ is finite for every $c \in C$, the functor $F$ can be encoded as a pair
of \texttt{Map}s: one mapping the identifiers of the objects $c$ of $C$ to the
encodings of the presheaves $F(c)$, and one mapping the identifiers of the maps
$f \co c \to c'$ of $C$ to encodings of the morphisms $F(f) \co F(c) \to F(c')$.

Given an encoding of such an $F$ and an encoding of a finite presheaf $X \in \ps
C$, by unfolding \Cref{def:kan-extension}, one can compute the image $\Lan F(X)$
of $X$ by $\Lan F$. Indeed, it can be computed as the coequaliser of
\[
  \begin{tikzcd}[csbo=11em]
    \coprod\limits_{\mathmakebox[18pt][c]{f \co c \to c' \in C}} (F(c) \otimes X(c'))
    \ar[r,shift left,"{[\iota_{c'} \circ (F(f) \otimes X(c'))]_{f : c \to c' \in C}}"]
    \ar[r,shift right,"{[\iota_{c} \circ (F(c) \otimes X(f))]_{f : c \to c' \in C}}"']
    &[6em]
    \coprod\limits_{c \in C} F(c) \otimes X(c)
  \end{tikzcd}
\]
which can be computed given the encodings of $F$ and $X$: one first compute the
encodings of the $F(c) \otimes X(c) \cong \coprod_{x \in X(c)} F(c)$ for $c \in
C$, and the full coproduct $\coprod_{c \in C} F(c) \otimes X(c)$ using the
discussed method for computing coproducts. And then, one compute the coequaliser
using the extension to presheaves of the method for sets based on the Union–Find
algorithm. Thus, we obtain an encoding of $\Lan F(X)$, and, with additional
administrative work, encodings of the coprojections $F(c) \otimes X(c) \to \Lan
F(X)$ for $c \in C$.

Now, given a morphism $f \co X \to Y$ between two finite presheaves $X$ and $Y$,
we compute the image of $\Lan F(f)$ by observing that we have the commutative
diagram
\[
  \begin{tikzcd}
    F(c) \otimes X(c)
    \ar[r,"{F(c) \otimes f_c}"]
    \ar[d,"p^c"']
    &
    F(c) \otimes Y(c)
    \ar[d,"q^c"]
    \\
    \Lan F(X)
    \ar[r,"\Lan F(f)"']
    &
    \Lan F(Y)
  \end{tikzcd}
\]
for every $c \in \Ob(C)$, where $p^c$ and $q^c$ are the canonical coprojections.
Since, as shown above, $\Lan F(X)$ is computed as a coequaliser, the elements of
the presheaf $\Lan F(X)$ are all image by $p^c$ of some element of $F(c) \otimes
X(c)$ for some $c \in \Ob(C)$. Moreover, if two elements of $F(c) \otimes X(c)$
and $F(c') \otimes X(c')$ respectively, for some $c,c' \in \Ob(C)$, are sent by
$p^c$ and $p^{c'}$ to the same element of $\Lan F(X)$, their images by $q^c
\circ (F(c) \otimes f^c)$ and $q^{c'} \circ (F(c') \otimes f^{c'})$,
respectively, will be sent to the same element of $\Lan F(Y)$. Thus, by
iterating over all the elements of all the $F(c) \otimes X(c)$ presheaves for $c
\in \Ob(C)$, we are able to compute an encoding of $\Lan F(f)$.

Thus, when $F$ is a Kan model between two presheaf models $(C,O)$ and $(D,P)$,
we are able to compute the action of $F' \co \ps C \to \ps D$ on particular
presheaves and morphisms of presheaves of $\ps C$.

\subsubsection{Computing reflection}

Given a finite presheaf model $(C,O)$, we must now consider the computation of
the moves of the game of reflection.

First, we might want to know what are the available moves of a certain type,
starting from a morphism $m \co X \to Y \in \ps C$ seen as a configuration of
the game. Consider the Dom-E moves relative to some orthogonality condition $g
\co A \to B \in \ps C$, for example. In order to know all the Dom-E moves
involving $g$, we must:
\begin{enumerate}[1)]
\item compute all the morphisms $f \co A \to X$
\item for each such $f \co A \to X$, compute all the $h \co B \to Y$ such that
  $h\circ g = m \circ f$
\item return the pairs $(f,h)$
\end{enumerate}
Note that one can compute all the functions between two finite sets $S$ and
$S'$. By extension, one can compute all the morphisms between two given finite
presheaves $Z$ and $Z'$ of $\ps C$: one just need to compute all the possible
families of functions $(k_c \co Z(c) \to Z(c))_{c \in \Ob(C)}$ and then filter
out the ones that do not satisfy the naturality condition of natural
transformations $Z \To Z'$, which is a computable condition on such family of
morphisms, based on the encodings of $Z$ and $Z'$. Thus, the above procedure can
be computed. Similarly, one can compute all the instances of the Dom-U,
Cod-E and Cod-U moves.

Once one knows all the possible moves of a certain type relative to a certain $g
\co A \to B \in O$, one might want to compute the result of applying the move.
We will again consider for example an instance $(f,h)$ of a Dom-E move, from
the list computed above. One must first compute the pushout
\eqref{eq:dom-moves-pushout}, which can be done, since every finite colimit of
finite presheaves can be computed as we explained earlier.

Once that such a pushout is computed, one must compute the factorisation $m' \co
X' \to Y$ of the cocone $(m,h)$. Since every element of $X'$ is the image of
either one element of $X$ or one element of $B$, the encoding of $m'$ can be
computed by iterating over the elements of $X$ and $B$, computing their images
in $X'$ and $Y$, and then adding the computed bindings to the definition of
$m'$. This way, we were able to compute the result $m'$ of a particular instance
$(f,h)$ of a Dom-E move. Computing the results of the other moves can be
similarly achieved.

These methods allow one to compute the result of moves one by one, but one can
adapt the procedure to compute the result of all the available moves at once.
This can be done by computing a colimit on a diagram similar
to~\eqref{eq:step-Xi-diag}.


\end{document}